\documentclass[11pt]{amsart}

\usepackage{setspace} 
\usepackage{amssymb}
\usepackage[dvips]{color}
\usepackage{amsmath}
\usepackage{amsthm}
\usepackage[a4paper]{geometry}
\usepackage{amsfonts}
\usepackage[all]{xypic} 
\usepackage{bbm}
\usepackage{xcolor}

\usepackage{color}
\usepackage[colorlinks=true,linkcolor=blue,citecolor=blue,pdfpagelabels=false]{hyperref}
\usepackage{cleveref}
\usepackage{url}

\usepackage{hyperref}

\usepackage[colorinlistoftodos]{todonotes}

\allowdisplaybreaks[4] 
\newtheoremstyle{myremark}     {10pt}{10pt}{}{}{\bfseries}{.}{.5em}{}
\newtheoremstyle{myremark}     {10pt}{10pt}{}{}{\bfseries}{.}{.5em}{}
\newtheorem{thm}{Theorem}[section]

\newtheorem{lem}[thm]{Lemma}
\newtheorem{prop}[thm]{Proposition}
\theoremstyle{definition}
\newtheorem{defn}[thm]{Definition}
\newtheorem{exmp}[thm]{Example}
\theoremstyle{myremark}
\newtheorem{rem}[thm]{Remark}

 \newcommand{\eps}{\varepsilon}

 \newcommand{\ri}{\rightarrow}

 \newcommand{\F}{\mathbb{F}}

 \newcommand{\N}{\mathbb{N}}
 \newcommand{\Z}{\mathbb{Z}}
 \newcommand{\E}{\mathbb{E}}

 \newcommand{\R}{\mathbb{R}}

 \newcommand{\abs}[1]{\left\vert#1\right\vert}

 \newcommand{\norm}[1]{\left\Vert#1\right\Vert}

\DeclareMathOperator{\id}{id}

\DeclareMathOperator{\tr}{tr}

\newcommand{\CP}{\mathcal{P}}

\begin{document}

\title{Convergence of noncommutative spherical averages for actions of free groups}

\author[Bikram]{Panchugopal Bikram}
\address{\hskip-\parindent
Panchugopal Bikram, School of Mathematical Sciences, National Institute of Science Education and Research,  Bhubaneswar, An OCC of Homi Bhabha National Institute, Jatni -  752050, India.}

\email{bikram@niser.ac.in}

%\author[De]{Debabrata De}
%\address{\hskip-\parindent
%Debabrata De, School of Mathematical Sciences, National Institute of Science Education and Research,  Bhubaneswar, An OCC of Homi Bhabha National Institute, Jatni -  752050, India.}

%\email{debabratade.ju@gmail.com}



\begin{abstract}
 In this article,  we   extend the Bufetov   pointwise ergodic theorem for  spherical averages of even radius for free group actions on noncommutative $L\log L$-space. Indeed, we  extend it  to  more general Orlicz space $L^\Phi(M, \tau)$ (noncommutative/classical), where $M$  is the semifinite von Neumann algebra with faithful normal semifinite  trace $\tau$ and  $\Phi: [0, \infty ) $ is  a Orlicz function such that $ [0, \infty ) \ni t\ri \left({\Phi(t)}\right)^{  \frac{1}{p}}$ is convex for some $p >1$. To establish this convergence we follow similar approach as Bufetov (\cite{Bu02}) and  Anantharaman-Delaroche (\cite{Ad06}).  Thus, additionally we obtain Rota theorem on the same noncommutative Orlicz space by  extending the  earlier work of Anantharaman-Delaroche. In \cite{Ad06}, author proved Rota theorem  for noncommutative $L^p$-spaces for $p >1$,  and mentioned as ``interesting open problem''   to extend it to  noncommutative $L\log L$-space as classical case. In the  end we also look at the  convergence of  averages of spherical averages associated to free group and free semigroup actions on  noncommutative spaces. 
\end{abstract}

\maketitle

\section{Introduction}\label{Sec:Intro}
This article is dedicated  to study ergodic theorems of the averaging operators associated to actions of  free group or free semigroup in noncommutative spaces (quantum setting). \\
The origin of ergodic theory on classical measure spaces traces back to the 1930s due to Birkhoff and von Neumann.  Since then, the subject went through substantial evolution.
There are numerous  ergodic theorems for the  actions of various  different groups on classical measure spaces (cf. \cite{anantharaman2010ergodic}, \cite{calderon1953general}). For example in  
\cite{Lindenstrauss2001}, Lindenstrauss proved  pointwise ergodic theorems for tempered F\o lner sequences, which is considered as the generalisation of Birkhoff's ergodic theorem for actions of second countable, amenable groups. 
 
The noncommutative analogue of the ergodic theorems first appeared in the works of Lance (\cite{Lance1976}) in 1976. In (\cite{Lance1976}), the author established a pointwise ergodic theorem for the averages of a single automorphism in von Neumann algebra that preserves a faithful, normal state.
 These results are then generalised by Kummerer, Conze, Dang-Ngoc for positive maps in the von Neumann algebra setting and many others (cf. \cite{Ku78}, \cite{Conze1978} and references therein) .\\
 In the seminal works \cite{Yeadon1977} and \cite{Yeadon1980}, Yeadon obtained maximal ergodic theorems for actions of positive subtracial, subunital maps on the preduals of semifinite von Neumann algebras and proved associated pointwise ergodic theorems to the noncommutative $L^1$-spaces.\\
 Later on Junge and Xu in \cite{JX07} in 2007, extended Yeadon's weak type  $(1,1)$ maximal ergodic inequality  in  the noncommutative $L^p$-spaces using interpolation techniques and  used it to prove an individual ergodic theorem for Dunford-Schwartz  operator in noncommutative $L^p$-spaces for $1< p < \infty$. \\
 In  \cite{Hong2021}, Hong, Liao and Wang made  a significant breakthrough by  obtaining  several ergodic theorems on noncommutative $L^p$-spaces for the actions of locally compact group metric measure spaces satisfying doubling conditions and compactly generated groups of polynomial growth in noncommutative $L^p$-spaces for $1< p < \infty$. 
 Furthermore, these results are extended to the general setting of amenable groups in \cite{cadilhac2022noncommutative}.  \\
It is worth mentioning here that both these articles deal with actions preserving a faithful, normal, semifinite trace. However, in \cite{JX07}, the authors studied maximal inequalities for ergodic averages in Haagerup's $L^p$-spaces ($1< p < \infty$) associated to state preserving actions of the group $\Z$ and semigroup $\R_+$. However,  ergodic theorems for state preserving actions in the endpoint case ($p=1$) was not studied until \cite{Bik-Dip-neveu}, \cite{BS23}. In, \cite{PgD24maximal-inequality}, authors obtained similar ergodic theorems for state preserving actions of locally compact group  metric measure space satisfying doubling condition  in noncommutative $L^1$-spaces.

%%%%%%%%%%%%%%%%%%%%%%%%%%%%%%%%%%%%%%%%%%%%%%%%%%%%%%%%%%%%%%%%%%%%%%%%%%%%%%%%%%%%%%%%%%%%%%%%%%%%%%%%%%%%%%%%%%%%%%%%%%%%%%%%%%%%%%%%%%%%%%%%%%%%%%%%%%%%%%%%%%%%%%%%%%%%%%%%%%%%%%%%%%%%%%%%%%%%%%%%%%%%%%%%%%%%%%%%%%%%%%%%%%%%%%%%%%%%%%%%

%%%%%%%%%%%%%%%%%%%%%%%%%%%%%%%%%%%%%%%%%%%%%%%%%%%%%%%%%%%%%%%%%%%%%%%%%%%%%%%%%%%%%%%%%%%%%%%%%%%%%%%%%%%%%%%%%%%%%%%%%%%%%%%%%%%%%%%%%%%%%%%%%%%%%%%%%%%%%%%%%%%%%%%%%%%%%%%%%%%%%%%%%%%%%%%%%%%%%%%%%%%%%%%%%%%%%%%%%%%%%%%%%%%%%%%%%%%%%%%%
In this article our main focus is to study 
the ergodic convergence of spherical averages  and averages of spherical averages associated with the actions of free group or semigroup  $\F_m$ of $m$-generators. 
The study of ergodic theorems for actions of groups that are not amenable was first initiated in the classical setting by Arnold and Krylov in \cite{AK63}. After that, it was carried out in great detail by many authors. We refer to \cite{grigorchuk1986individual}, \cite{Gui69}, \cite{Ne94}, \cite{NS94} and references therein. \\

The ergodic theorem for actions by arbitrary countable groups were established by V. I. Oseledets (\cite{Os65}) in the following setting.\\
Let $ G$ be a  countable  group and $(X, \nu)$ be a probability measure space. For each $g \in G$, let $ T_g: X\ri X$ be a measure preserving transformation. Suppose $ \mu$ be a probability measure on $G$ with $ \mu( g)= \mu( g^{-1})$ for all $ g \in G$ and let $ \mu^{ (n)}$ denotes the $n$-th convolution of the measure $\mu$. Then Oseledetes ergodic theorem says that the averages 
$$  A_{ 2n}(f) = \sum_{ g \in G} \mu^{ (2n)}(g) T_g\circ f, \text{ for } f \in L\log L( X, \nu), $$
converges pointwise almost everywhere. Then in 1969 in \cite{Gui69},  the author considered free group $\F_m$ and its action $(T_g)$ on a measure space $ (X, \nu)$, then proved that the following spherical average
$$ S_n (f)= \frac{1}{2m ( 2m-1)^{n-1}} \sum_{ g \in \mathbb{S}_n } T_g\circ f \text{ for } f \in L^2(X, \nu), $$
converges in $L^2$-norm to an $ \F_m^{(2)}$-invariant function,, where $\mathbb{S}_n $ denotes the set of all words of length $n$.

Later, Nevo and Stien  $($see \cite{NS94}, \cite{Ne94}$)$ considered the uniform spherical averages for measure-preserving actions of free groups and proved that the even-average converges pointwise for $L^p$ spaces, for $p>1$. Moving from uniform spherical average to non-uniform was challenging, which is demonstrated by Bufetov in his remarkable work \cite{Bu02}, which generalized the Nevo and Stien theorem \cite{NS94} for the Orlicz space $L\log L$, which includes all $L^p$, $p>1$. The proof uses the Markov operator approach and a suitable maximal inequality.
In fact  in  this setting, in 1994 Nevo and Stein obtained the following.

\begin{thm}[Nevo and Stein \cite{NS94}] Suppose $ p > 1 $, then for all $ f \in L^p(X, \nu)$, the sequence $(  S_{2n}(f)) $ converges pointwise almost everywhere in $L^p(X, \nu)$ to an $\F_m^{ (2)}$-invariant function. 
\end{thm}
After a while in 2002,  Bufetov proved the following result using the Rota's theorem. Indeed, Bufetov proved it in more general setting which we discuss in the sequel.  
\begin{thm}[Bufetov \cite{Bu02}] Let $ f \in L \log L( X, \nu) $, then the sequence $(  S_{2n}(f)) $ converges pointwise almost everywhere in $L^1(X, \nu)$ to an $\F_m^{ (2)}$-invariant function. 
\end{thm}

In the quantum setting, the free group action  was first considered by Walker in \cite{Walk:1997}.  To discuss his result we need to set the background. \\
Suppose the free group  $\F_m$  is  generated by  
$\{ a_1, a_2, \cdots, a_m \}$ and $M$ be a  von Neumann algebra. For $ 1 \leq i \leq m $,  let  $ \alpha_i : M \ri M $ be a normal  automorphism  corresponding  to generator $a_i$ and we set $ \alpha_{-i} = \alpha_i^{-1} $.
Let $ \omega \in \F_m $ and   $ \omega = a_{i_1} a_{i_2}\cdots a_{i_n}$, we consider
$$ \alpha_\omega = \alpha_{i_n} \circ \cdots \circ \alpha_{ i_1}.$$
Suppose $$ \mathbb{S}_n = \{ \omega \in \F_m :  \abs{ \omega} = n \},$$
where $ \abs{\cdot }$ denotes the length function. Then the averaging operator $S_n$ on the sphere $\mathbb{S}_n$ is defined as 
$$ S_n = \frac{1}{\abs{\mathbb{S}_n}} \sum_{ \omega  \in \mathbb{S}_n }\alpha_\omega. $$

Further, consider $ M_n =  \frac{1}{n}\underset{ k = 0}{ \overset{ n-1}{ \sum}} S_n $ for $n \in \N$. 
We refer $S_n$ as spherical average and $M_n$ as average of spherical averages. 
In this article, we mainly study pointwise  convergence of the sequence $(M_n)$ and $(S_n)$ . 

In \cite{Walk:1997},
the author proved that if  $\rho \circ \alpha_i = \rho $ for $ 1 \leq i \leq m$ where $\rho$   is  a faithful, normal state $\rho$ on $M$, then for all $x \in M$ the sequence $M_n(x)$ will converge almost uniformly to an element  $\hat{x}$ in $M$. In fact, Walker dealt with more general maps on $M$. In particular, the author considered a sequence $(T_n)$ of completely positive maps on $M$ which preserves a faithful, normal state $\rho$ and satisfies $T_1 \circ T_n= w T_{n+1} + (1-w) T_{n-1}$ for some $w \in (1/2,1]$ with $T_0(x)=x$. Then the author  proved the almost uniform convergence of the average $\frac{1}{n} \underset{k=0}{ \overset{n-1}{\sum}} T_k$ under some natural spectral condition on $T_1$. In \cite{ CLS05}, Walker's result was extended to $L^1(M, \tau)$, where $\tau$ is a faithful, normal, semifinite trace such that $\tau(T_1(x)) \leq \tau(x)$ for all $x \in M \cap L^1(M, \tau)$ with $0 \leq x \leq 1$.

In 2006,  Anantharaman-Delaroche  extended  Nevo and Stein ergodic theorem in  the quantum setting  in \cite{Ad06}. 
Anantharaman-Delaroche obtained 
 the following results for state preserving factorable  Markov maps. For the  definition of factorable map see the Definition \ref{facto}. 
\begin{thm}[Claire Anantharaman-Delaroche \cite{Ad06}] Suppose $p > 1$ and  $(M,\varphi)$ be a noncommutative probability space.
	For each  $1 \leq  i \leq m  $,    $\alpha_i: M \rightarrow M,~~~$ is  a factorizable $\varphi$-preserving Markov operators.   Then for $ x \in L^p(M, \varphi)$, 
	
	\begin{enumerate}
		\item 
		the sequence  $\big(S_{2n}(x) \big)$ converges in  \emph{b.a.u} to $ \E^{ (2)}(x)$ in $L^p(M, \tau)$, 
		\item the sequence $\left( \frac{1}{n}\underset{k =0}{ \overset{ n}{ \sum}} S_n(x) \right)$ converges to 
		\emph{b.a.u} to $ \E^{ (2)}(x)$ in $L^p(M, \varphi)$.
	\end{enumerate} 
	where $\E^{ (2)}$ is the conditional expectation  onto  the fixed points of $ \F_m^{(2)}$.
	Moreover, if $p \geq 2$, the above averages converges in  \emph{a.u}.
\end{thm}
 In \cite{Hu08}, Hu  also obtain  the Nevo-Stain pointwise ergodic theorem  for  spherical averages for free group actions  in noncommutative $L^p$-spaces $( p > 1)$ by generalizing the Nevo-Stein idea in the noncommutative  setting. Further, the same author in \cite{Hu09},  improved the regularity of the space and proved the  pointwise ergodic theorem  for the  spherical averages of free group action  to noncommutative $L\log^\beta L $-spaces for $ \beta \geq 2$.
 
Later on, inspired by the seminal works of Junge and Xu  (\cite{JX07}), in \cite{Hu08}, Hu obtained a maximal ergodic theorem and associated individual ergodic theorem in  $L^p(M, \tau)$ corresponding to $(T_n)$. In particular, the author considered the sequence $(T_n)$ as described above on $(M, \varphi)$, where $\varphi$ is a faithful, normal state, together with $\varphi \circ T_n \leq \varphi$ and $T_1 \circ \sigma_t^\varphi = \sigma_t^\varphi \circ T_1$. Under these assumptions, Hu proved maximal inequality and individual ergodic theorem of the average $ \frac{1}{n} \underset{ k=0}{\overset{ n}{\sum}} T_k $ in Haagerup noncommutative $L^p$-spaces for $1<p< \infty$. Furthermore, if $M$ is assumed to be semifinite with faithful, normal, semifinite trace and  $\tau(T_1(x)) \leq \tau(x)$ for all $x \in M \cap L^1(M, \tau)$ with $0 \leq x \leq 1$, then Hu recovered the result in \cite{CLS05} as mentioned above. In \cite{BS23}, authors  assumed that $\varphi \circ T_1 = \varphi$ and $\rho(  y T_1(x) ) = \rho(T_1(y) x ) $ for $ x, y \in M$ and proved individual ergodic theorem for the sequence $( \frac{1}{n} \underset{ k=0}{\overset{ n}{\sum}} T_k )$.\\

In this article, we discuss the individual convergence of the sequence $( \frac{1}{n} \underset{ k=0}{\overset{ n}{\sum}} T_k )$ in more general setting. 
Our main result in this article  is to establish  the noncommutative Bufetov pointwise ergodic theorem. Infact,  we not only prove the   Bufetov pointwise ergodic theorem for  free group actions on noncommutative $L\log L$-space, but also prove it for more general noncommutative Orlicz spaces.  Here we state our main results and for the definition of \emph{b.a.u} and \emph{a.u} see the Definition  \ref{Defn:AUconvergence}. Further, for definition of factorable map see the Definition \ref{facto}.

\begin{thm}[\textbf{ Non-commutative Bufetov ergodic theorem}]\label{Ergo-Orlicz}
	Let $M$ be a von Neumann algebra equipped with a faithful normal semfinite trace $\tau$ and
	 $\Phi: [0, \infty ) $ is  a Orlicz function such that $ [0, \infty ) \ni t\ri \left({\Phi(t)}\right)^{  \frac{1}{p}}$ is convex for some $p >1$. 
	For each  $1 \leq  i \leq m  $, suppose    $\alpha_i:(M,\tau)\rightarrow (M,\tau),~~~$ is  a factorizable $\tau$-preserving Markov operators.   Then we have the following. 
	
	\begin{enumerate}
		\item 
		For any  $x\in L^\Phi(M, \tau)$, the sequence  $\big(S_{2n}(x) \big)$ converges in  \emph{b.a.u} to $ \E^{ (2)}(x)$.
		\item In addition suppose the function $\widetilde{ \Phi } (t) = \Phi( \sqrt{t}) $ for $ t \in [1, \infty) $ is also $p$-convex. Then for  any  $x\in L^\Phi(M, \tau)$, the sequence  $\big(S_{2n}(x) \big)$ converges  in  \emph{a.u} to $ \E^{ (2)}(x)$,
	\end{enumerate}
 where $ \E^{ (2)}$ is defined prior to the  Theorem \ref{Thm-Ergo-Orlicz}. 
	In particular, $(1)$ holds for all $ x \in L \log L(M, \tau) $, i.e, the sequence $\big(S_{2n}(x) \big)$ converges in  \emph{b.a.u} for all $ x \in L \log L(M, \tau) $.
\end{thm}
To prove this theorem we follow Anantharaman-Delaroche's approach and prove a noncommutative version of  Rota's ``Alternierende Verfahren" theorem. Then  use  it to establish     Bufetov ergodic theorem in noncommutative  Orlicz spaces.

In function theory, many mathematicians in the $20^{th}$ century studied the limiting theorem involving the limit of products $T_1T_2\cdots T_n f$ of a sequence of certain operators $T_n$ on a measure space $(X,\mu)$. This study of convergence is known as the ``Alternierende Verfahren" theorem. G. C. Rota \cite[Theorem 1]{Ro62} investigated the convergence of these products involving doubly stochastic operators and their adjoints on $L^p(\mu)$-spaces, where $\mu$ is a probability measure. 
\begin{thm}[\cite{Ro62}]
	Let $T$ be a $\mu$-preserving Markov operator on the probability space $(X,\mu)$. Then for every $p>1$ and $f\in L^p(X,\mu)$, the sequence $T^n (T^*)^n (f)$ converges almost everywhere.
\end{thm}

% \begin{thm}[\textbf{Bufetov Theorem}]\label{Thm:Bufetov}
	%   To write 
	% \end{thm}

Then Anantharaman-Delaroche first proved the noncommutative analog of Rota's theorem for factorizable Markov operators in noncommutative $L^p$-spaces, where $p>1$  in \cite{Ad06}. The proof relies on a noncommutative Kolmogorov construction and Junge's Doob maximal inequality for noncommutative martingales \cite{Ju02}.\\
Anantharaman-Delaroche mentioned (see \cite{Ad06}) that \emph{``the extension of Rota's theorem to noncommutative setting (noncommutative $L\log L$-spaces) is an interesting open problem"}.\\
Ying Hu later considered a factorable operator and proved Rota's theorem for noncommutative Orlicz space $L\log^\alpha L$, where $\alpha\geq 2$, by establishing a maximal inequality for noncommutative martingales from $L\log^\alpha L$ to $L^1$ \cite{Hu09}. They also showed that the exponent $2$ is optimal for that maximal inequality. Conde-Alonso et al. recently proved a similar result for a bi-parametric case involving the noncommutative Orlicz space $L\log^2 L$ and $L^1$ \cite{CGP20}. In this article, we fully establish the noncommuttaive version of Rota's theorem not only for noncommutative $L\log L$-spaces but also for more general noncommutative Orlicz spaces. Indeed, we prove the following theorem.

\begin{thm}[\textbf{Noncommutative Rota's theorem}]\label{NCRotaOrlicz}
	Let $M$ be a von Neumann algebra equipped with a faithful normal tracial state $\tau$. Suppose   $T:(M,\tau)\rightarrow (M,\tau)$ is  a factorizable $\tau$-preserving Markov operator and  $\Phi: [0, \infty ) $ is  a Orlicz function such that $ [0, \infty ) \ni t\ri \left({\Phi(t)}\right)^{  \frac{1}{p}}$ is convex for some $p >1$.
	Then  for any $x\in L^\Phi(M, \tau)$, the sequence  $\big(T^n (T^*)^n( x) \big)$ converges in \emph{b.a.u.}\\
	Furthermore, if the function $ \widetilde{ \Phi }(t) = \Phi( \sqrt{ t})$ for $ t \in [0, \infty )$, is $ p$-convex, then for all $x\in L^\Phi(M, \tau)$, the sequence  $\big(T^n (T^*)^n( x) \big)$ converges in \emph{a.u.}\\
	 In particular, , i.e, the sequence the sequence  $\big(T^n (T^*)^n( x) \big) $ converges in  \emph{b.a.u} for all $ x \in L \log L(M, \tau) $.\\
\end{thm}
Our approach employs Anantharaman-Delaroche's noncommutative Kolmogorov construction/ noncommutative Markov chain construction  \cite{Ad06} and then prove the bilateral uniform equicontinuity in measure of a sequence of martingales in Orlicz space setting.  
In general, a maximum inequality is required for establishing pointwise convergence of a sequence of maps. But,  proving a maximal inequality, often requires highly delicate analysis that can occasionally become complex. However, it has recently been shown that in noncommutative settings, bilateral uniform equicontinuity in measure is highly helpful in establishing the pointwise convergence of a sequence  of maps (see \cite{Li12}, \cite{CL17}, \cite{CL06},\cite{ CLS05}, \cite{pg-ds-ergo-orlicz}  and references therein).

Now we describe the lay out of this article. \S \ref{Sec:preli} is the preliminary section. In the preliminary section, we recollect some standard facts of the literature which will be useful in  this article. In \S \ref{Rota}, we prove noncommutative Rota Theorem. In the last section, we study the convergence of various averages  associated to free group or semigroup actions on noncommutative spaces. In particular in this section  we prove  the pointwise convergence of spherical averages of even radius associated to free group actions on noncommutative Orlicz spaces.

\section{Preliminaries}\label{Sec:preli}

In this section, we  set preliminary background of this article, we briefly recollect some Tomita-Takesaki theory $($for details, see \cite{SZ19}, \cite{St20, Ta03}$)$,  noncommutative $L^p$-spaces $($for details, see \cite{Hi21, GL20}$)$, Haagerup $L^p$-spaces, noncommutative Orlicz spaces and vector valued $L^p$-spaces. We also discuss some technical results regarding  bilateral and almost uniform uniform convergence of a sequence of maps.

In this article, all inner products are linear in the left variable, all Hilbert spaces are separable, and all von Neumann algebras have separable predual. 
Further, $M$ denotes a separable von Neumann algebra. Suppose $ \varphi$ is a faithful normal (in short f.n) state or weight (resp. finite f.n trace or f.n semifinite trace) on $M$, then we refer the pair $( M, \varphi) $ as  noncommutative measure space (resp. noncommutative tracial measure space) and if $\varphi $ is f.n state (resp. tracial state), then refer the pair $( M, \varphi) $   as  noncommutative probability space (resp. noncommutative tracial probability space).

Let us first  recall some basic facts on standard form of von Neumann algebra $M$ when it equipped with a faithful normal state $\varphi$. For weight we refer to \cite{St20}.  Let $M$ be represented on the GNS Hilbert space $\mathcal{H}_{\varphi}:=L^2(M,\varphi)$
in standard form \cite{Ha75}, the GNS vector being denoted by $\Omega_{\varphi}$. 
The inner product and norm on $\mathcal{H}_{\varphi}$ will respectively be denoted by 
$\langle\cdot,\cdot\rangle_{\varphi}$ and $\norm{\cdot}_{2,\varphi}$, and $\norm{\cdot}$
will denote the operator norm of $M$. Recall that  $\mathfrak{A}_{\varphi}=M\Omega_{\varphi}$ is the canonical full left Hilbert algebra associated with $\varphi$. The involution on $\mathfrak{A}_\varphi$ is given by 
$\mathfrak{A}_\varphi\ni x\Omega_\varphi\mapsto x^*\Omega_\varphi\in \mathfrak{A}_\varphi$ 
which is  closable and its closure is denoted by $S_\varphi$. 

Let $\Delta_{\varphi}$, $J_\varphi$ and $(\sigma_t^{\varphi})$ respectively denote 
the Tomita's modular operator, conjugation operator and the modular automorphism 
group associated with $\varphi$. Then $\Delta_\varphi$ is positive, nonsingular 
and self-adjoint, and $S_\varphi=J_\varphi\Delta_\varphi^{1/2}$ is the polar 
decomposition of $S_\varphi$. Moreover, $\Delta_\varphi^{it}J_\varphi
=J_\varphi\Delta_\varphi^{it}$ for all $t\in\mathbb{R}$, and in general 
$J_\varphi f(\Delta_{\varphi})J_\varphi = \bar{f}(\Delta_{\varphi}^{-1})$ for all 
complex valued Borel functions $f$ on $[0,\infty)$. Note that $t\mapsto \Delta_\varphi^{it}$ 
is $s.o.t.$ continuous group of unitaries and is called the modular group associated with 
$\mathfrak{A}_\varphi$ and $(Ad \text{ }\Delta_\varphi^{it})_{\upharpoonleft M}
=\sigma_t^\varphi$ for all $t\in\mathbb{R}$.

\subsection{Noncommutative $L^p$-spaces.} In this subsection, we recall noncommutative $L^p$-spaces associted to tracial von Neumann algebras  and some of their basics properties. Most of the basic properties of noncommutative $L^p$-spaces can be found in \cite{Hi21, GL20}. 

Let $M\subseteq \mathbf{B}(\mathcal{H})$ be a semifinite von Neumann algebra equipped with a faithful normal semifinite $($f.n.s. in the sequel$)$ trace. Let $L^0(M, \tau)$ denote the set of all $\tau$-measurable operators.  The semifinite trace $\tau$ can be extended to the positive cone $L^0_+(M, \tau)$ of $L^0(M, \tau)$ as follows:
\begin{align}
    \tau(x)=\int_0^\infty \lambda \, d\tau(e_\lambda),
\end{align}
where $x=\int_0^\infty \lambda \,de_\lambda$ denotes the spectral decomposition of $x$.  For $0\leq p<\infty$ and $x\in L^0(M, \tau)$, define $\norm{x}_{p}=\tau(\abs{x}^p)^{1/p}$. Then, the noncommutative $L^p$-space $L^p(M)$ associated with the pair $(M,\tau)$ is defined as:
\begin{align*}
L^p(M)=\{x\in L^0(M,\tau): \norm{x}_p<\infty\}.
\end{align*}
Note that $L^p(M)$ equipped with the norm $\norm{\cdot}_{p}$ becomes a Banach space for $p\geq 1$ and a quasi-Banach Space for $p<1$. As usual, for $p=\infty$,  we set $L^\infty(M)=M$ equipped with the operator norm.

For $x\in L^0(M, \tau)$ and $t>0$, the $t$-th generalized $s$-number $\mu_t(x)$ defined by
\begin{align*}
\mu_t(x):&=\inf\{s\geq 0: \tau(e_{(s,\infty)}(\abs{x}))\leq t\}\\
&=\inf\{\norm{xe}_\infty: e\in \mathcal{P}(M), \tau(e^\perp)\leq t\}.
\end{align*} 
The map $[0,\infty)\ni t\rightarrow \mu_t(x)$ is known to be non-increasing and right-continuous. For every $x\in L^0_+(M, \tau)$,
\begin{align*}
\tau(x)=\int_0^\infty \mu_t(x)\, dt.
\end{align*}
Moreover, for any continuous non-decreasing function $f$ on $[0,\infty)$ with $f(0)\geq 0$,
\begin{align*}
\tau (f(x))=\int_0^\infty \mu_t(f(x))\, dt=\int_0^\infty f(\mu_t(x))\, dt.
\end{align*}

Let us recall the $K$-functional in real interpolation theory, which will be useful for our purpose. This was proven in \cite{FK86} (see also \cite{BL76}).

\begin{prop}\label{Prop:IntegralDecRea}
	For $x\in L^0(M, \tau)$ and $t>0$, we have
	\begin{align*}
		\int_0^t \mu_s(x)\,ds=\inf\left\{ \tau(\abs{y})+t\norm{z}_\infty: y,z\in L^0(M, \tau), \,x=y+z\right\}.
	\end{align*}
\end{prop}
Another theorem that is helpful for our purpose is proven by Hardy, Littlewood, and Polya \cite{HLP29} is given below.
\begin{thm}\label{Thm:HLP}
	Let $f,g$ be non-negative Borel measurable functions on $[0,\infty)$ which are finite almost everywhere. Then the following are equivalent:
	\begin{enumerate}
		\item $\int_{0}^t f(s)\, ds\leq \int_0^t g(s)\, ds$ for all $t>0$.
		\item any non-negative, non-decreasing, convex continuous function $\Phi: [0,\infty)\rightarrow [0,\infty)$, which is not identically zero, satisfying the following inequality:
		\begin{align*}
			\int_{0}^t \Phi(f(s))\, ds\leq \int_0^t \Phi(g(s))\, ds.
		\end{align*}
	\end{enumerate}
\end{thm}
\subsection{Haagerup $L^p$-spaces.} 

Let us now discuss on the Haagerup $L^p$-space associated with the pair $(M,\varphi)$. Let $\mathcal{R}=M\rtimes_{\sigma}\mathbb{R}\subseteq \mathbf{B}(L^2(\mathbb{R},\mathcal{H}_\varphi))$ denote the crossed product of $M$ by $\mathbb{R}$ with respect to the modular automorphism group $\sigma:=\sigma^{\varphi}$ is the von Neumann algebra generated by the elements $\pi(x)$, $x\in M$ and $\lambda(t)$, $t\in\mathbb{R}$, on $L^2(\mathbb{R},\mathcal{H}_\varphi)$, where
\begin{align*}
	\pi(x)\xi(s)=\sigma_{-s}(x)\xi_{s}\quad \text { and }\quad \lambda(t)\xi(s)=\xi(s-t).
\end{align*}
The action $\sigma$ of $\mathbb{R}$ admits a dual action $\widehat{\sigma}$ of the dual 
group $\widehat{\mathbb{R}}=\mathbb{R}$ on $\mathcal{R}$ which satisfies
\begin{align*}
	\widehat{\sigma}_{t}(\pi_{\sigma}(x))=\pi_{\sigma}(x), \text{ and, } \widehat{\sigma}_{t}(\lambda(s))=\mathrm{e}^{-ist}\lambda(s),\text{ }\text{ }x\in M,\text{ }t\in \widehat{\mathbb{R}}, \text{ }s\in\mathbb{R}.
\end{align*}
It is well known that 
$\mathcal{R}$ is semifinite and there exists a unique canonical $f.n.s.$ tracial weight $\tau$ on $\mathcal{R}$ satisfying $\tau\circ\widehat{\sigma}_t=\mathrm{e}^{-t}\tau$, $t\in\mathbb{R}$. 

Let $\mathcal{A}(\mathcal{R})$ denote the collection of all $\tau$-measurable operators affiliated to $\mathcal{R}$. Then, $\mathcal{A}(\mathcal{R})$ becomes a complete Hausdorff topological $*$-algebra with respect to strong sum and strong product in which $\mathcal{R}$ is dense. Note that the dual action $\widehat{\sigma}$ extends to $\mathcal{A}(\mathcal{R})$ and this extension will also be denoted by $\widehat\sigma$ with slight abuse of notation.

Let $\mathcal{A}(\mathcal{R})$ denote the collection of all $\tau$-measurable operators affiliated to $\mathcal{R}$. Then, $\mathcal{A}(\mathcal{R})$ becomes a complete Hausdorff topological $*$-algebra with respect to strong sum and strong product in which $\mathcal{R}$ is dense. Note that the dual action $\widehat{\sigma}$ extends to $\mathcal{A}(\mathcal{R})$ and this extension will also be denoted by $\widehat\sigma$ with slight abuse of notation. Following \cite{Te81, Ha77}, for $1\leq p\leq\infty$, define
\begin{align*}
	L^p(M,\varphi)=\{x\in \mathcal{A}(\mathcal{R}): \widehat\sigma_t(x) =\mathrm{e}^{-\frac{t}{p}} x\text{ }\forall \text{ }t\in\mathbb{R}\}.
\end{align*}
For $\psi\in M_{*}^{+}$, we denote by $D_{\psi}:=\frac{d\tilde{\psi}}{d\tau}$, the Radon-Nikodym derivative of $\tilde{\psi}$ with respect to $\tau$ $(\tilde\psi$ being the associated dual weight$)$. Consider $L^{1}(M,\varphi)$ equipped with the inherited topology coming from $\mathcal{A}(\mathcal{R})$.

Let $\mathcal{A}(\mathcal{R})$ denote the collection of all $\tau$-measurable operators affiliated to $\mathcal{R}$. Then, $\mathcal{A}(\mathcal{R})$ becomes a complete Hausdorff topological $*$-algebra with respect to strong sum and strong product in which $\mathcal{R}$ is dense. Note that the dual action $\widehat{\sigma}$ extends to $\mathcal{A}(\mathcal{R})$ and this extension will also be denoted by $\widehat\sigma$ with slight abuse of notation. Following \cite{Te81, Ha77}, for $1\leq p\leq\infty$, define
\begin{align*}
	L^p(M,\varphi)=\{x\in \mathcal{A}(\mathcal{R}): \widehat\sigma_t(x) =\mathrm{e}^{-\frac{t}{p}} x\text{ }\forall \text{ }t\in\mathbb{R}\}.
\end{align*}
For $\psi\in M_{*}^{+}$, we denote by $D_{\psi}:=\frac{d\tilde{\psi}}{d\tau}$, the Radon-Nikodym derivative of $\tilde{\psi}$ with respect to $\tau$ $(\tilde\psi$ being the associated dual weight$)$. Consider $L^{1}(M,\varphi)$ equipped with the inherited topology coming from $\mathcal{A}(\mathcal{R})$.

Then, the map
\begin{align}\label{Eq:PredualBijectionL1}
	M_{*}^{+}\ni \psi\mapsto D_{\psi}=\frac{d\tilde{\psi}}{d\tau}\in L^{1}(M,\varphi)^{+},
\end{align}
extends to a linear homeomorphism of $M_{*}$ onto $L^{1}(M,\varphi)$. This, allows one to equip $L^1(M,\varphi)$ with a norm defined by
\begin{align*}
	\norm{D_\psi}_{1}:=\abs{\psi}(1)=\norm{\psi}_{M_*},\text{ }\psi\in M_{*}. 
\end{align*}
Consequently, the map considered in Eq. \eqref{Eq:PredualBijectionL1} is an isometric isomorphism between $(M_*,\norm{\cdot}_{M_*})$ and $(L^1(M,\varphi),\norm{\cdot}_1)$.  Thus, $L^{1}(M,\varphi)$ equipped with $\norm{\cdot}_{1}$ is a Banach space.

Note that 
\begin{align*}
	\tr: L^1(M,\varphi)\rightarrow\mathbb{C} \text{ defined by }\tr(D_\psi)=\psi(1),\text{ }\psi\in {M_{*}},
\end{align*}
is linear, positive and $\norm{\cdot}_1$-continuous. In fact, it is a contraction. 

If $x\in L^p(M,\varphi)$, then $\abs{x}^p\in L^1(M,\varphi)$ for $1\leq p <\infty$ \cite[Prop. 12, pp. 39]{Te81}. 
Consequently, for $1< p<\infty$, $L^{p}(M,\varphi)$ equipped with the norm $\norm{\cdot}_{p}$ defined by
\begin{align*}
	\norm{x}_{p}=\tr(\abs{x}^{p})^{\frac{1}{p}}, \text{ }x\in L^{p}(M,\varphi),
\end{align*}
becomes a Banach space. 

Observe that $L^{\infty}(M,\varphi)$ coincides with $M$. Therefore, $\norm{\cdot}_{\infty}$ on $L^{\infty}(M,\varphi)$ is defined to be the usual operator norm on $M$.

Fix $1\leq p<\infty$ and $x\in M$. Then, the operators $x D_{\varphi}^{\frac{1}{p}}$ and 
$D_{\varphi}^{\frac{1}{2p}}x D_{\varphi}^{\frac{1}{2p}}$ defined on $L^2(\mathbb{R},\mathcal{H}_\varphi)$ 
are closable \cite{Wa88}, and we denote the closure of these operators by $x D_{\varphi}^{\frac{1}{p}}$ and
$D_{\varphi}^{\frac{1}{2p}}x D_{\varphi}^{\frac{1}{2p}}$ respectively with slight abuse of notation. Note that  
$L^{p}(M,\varphi)=\overline{MD_{\varphi}^{1/p}}^{\norm{\cdot}_{p}}$, where $MD_{\varphi}^{1/p}:=\{yD_{\varphi}^{1/p}:y\in M\}$ \cite[Lemma 3]{Wa88}.  Further, by \cite[Cor. 4]{Wa88}, $L^{p}(M,\varphi)=\overline{D_{\varphi}^{\frac{1}{2p}}M D_{\varphi}^{\frac{1}{2p}}}^{\norm{\cdot}_{p}}$, where $D_{\varphi}^{\frac{1}{2p}}M D_{\varphi}^{\frac{1}{2p}}:=\{D_{\varphi}^{\frac{1}{2p}}y D_{\varphi}^{\frac{1}{2p}}:y\in M\}$.

\subsection{Noncommutative Orlicz spaces.} In this section, we discuss the preliminaries of Orlicz spaces. Details can be found in \cite{GL20}.

Let $\Phi:[0,\infty)\rightarrow [0,\infty)$ be an Orlicz function, i.e., a continuous increasing and convex function satisfying $\Phi(0)=0$ and $\lim_{t\rightarrow\infty} \Phi(t)=\infty$. We say that  a Orlicz function $ \Phi$ satisfies the $\Delta_2$ condition if there exists a constant $ K > 0$ such that 
 $$ \Phi(2 t ) \leq K \Phi(t), \text{ for all } t >0.$$
In this article we assume that  $\Phi$ satisfies $\Delta_2$-condition  and we write $\Phi\in\Delta_2$ (unless otherwise stated).

\begin{defn}
Given a semifinite von Neumann algebra $(M,\tau)$ and an Orlicz function $\Phi$, the noncommutative Orlicz space $L^\Phi(M, \tau)$ is defined by the collection of all $x\in L^0(M, \tau)$ for which there exists some positive constant $\lambda>0$ such that $\tau\left(\Phi(\lambda^{-1} \abs{x})\right)<\infty$. The space $L^\Phi(M)$, equipped with the norm
\begin{align*}
    \norm{x}_\Phi=\inf\{ \lambda>0: \tau\left(\Phi\left(\lambda^{-1} \abs{x}\right)\right)\leq 1\},
\end{align*}
is a Banach space. Further, it is known that 
$$ 
\norm{x}_\Phi = \inf \{ \eps >0  :    \int_0^\infty \Phi \left( \frac{ \mu_t (x) }{ \eps} \right) \leq 1   \}.
$$
Equivalently, we have $\norm{ x}_\Phi = \norm{ \left(\mu_t(x) \right) }_\Phi$, where $\norm{ \left(\mu_t(x) \right) }_\Phi$ denotes the Orlicz norm of the function $ [0, \infty )\ni t  \ri \mu_t(x) $.
We remark that if $ \Phi \in \Delta_2$, then for  $ x \in L^0( M, \tau)$,  $~~\, \tau(  \Phi(\abs{x})) < \infty $ if and only if $ x \in L^\Phi( M, \tau)$. We further note that 
the  function $ \Phi_s(t) = t (\log(1+t))^s $ for $ s \geq 0$  is an Orlicz function and it satisfies $\Delta_2$ condition.  
\end{defn}
Following lemma will be useful in the sequel. We add a proof for plain reading.
\begin{lem}\label{leq}
	Suppose $ \Phi$ be a Orlicz function. Suppose $  x \in L^\Phi(M, \tau) $ with $ \norm{ x }_\Phi \leq 1$. Then we have the following  $$ \tau( \Phi(x)) \leq \norm{ x }_\Phi.$$
\end{lem}

\begin{proof}
We may assume $\norm{ x }_\Phi >0$. 	First we note that $ \norm{ x }_\Phi = \norm{ \left( \mu_t(x) \right) }_\Phi$.
Let $ (r_n)	\subseteq ( \norm{ x }_\Phi, \infty )$ be a decreasing sequence such that 
$ \underset{ n \ri \infty }{ \lim } r_n = \norm{ x}_\Phi $. Then by monotone convergence theorem it follows that 
$$  \underset{ n \ri \infty }{ \lim } \int_0^\infty \Phi\left( \frac{\mu_t(x)}{r_n} \right)dt = \int_0^\infty \Phi\left( \frac{\mu_t(x)}{ \norm{ x }_\Phi} \right)dt $$
We note that $ \int_0^\infty \Phi\left( \frac{\mu_t(x)}{r_n} \right)dt  \leq 1$ for all $n \in \N$.  Therefore, we note that $\int_0^\infty \Phi\left( \frac{\mu_t(x)}{ \norm{ x }_\Phi} \right)dt \leq 1$. Consequently we have 
$$ \tau \left(  \frac{\Phi(x)}{ \norm{x}_\Phi} \right) \leq 1.$$
As $ \Phi$ is convex  and $ \Phi(0) = 0$,  so we have $ \Phi( rt ) \leq r \Phi(t) $ for all $  0 \leq r \leq 1 $ and $t \geq 0$.  Consequently, $ u \Phi(t) \leq  \Phi(  u t )  $ for $ u \geq 1 $ and $ t \geq 0$. Since $ 0 < \norm{ x }_\Phi \leq 1 $, we have 
$$  \frac{1}{\norm{ x }_\Phi } \Phi(t) \leq  \Phi \left(    \frac{t}{\norm{ x }_\Phi} \right). 
$$
Since $ 0 \leq \norm{ x } \leq 1 $, we have $  \frac{1}{\norm{ x }_\Phi } \Phi(x) \leq  \Phi \left(    \frac{x}{\norm{ x }_\Phi}\right)$. Thus, we have 
$$ \tau( \Phi(x))  \leq \norm{ x}_\Phi \tau \left(  \frac{\Phi(x)}{ \norm{x}_\Phi} \right) \leq \norm{ x }_\Phi.$$
\end{proof}

\begin{defn}\label{Defn:AUconvergence}
	Let $M$ be a von Neumann algebra equipped with a $f.n $ semifinite  trace $\tau$.  A sequence $(x_n)$ in $L^0(M,\tau)$ is said to be converges to $x\in L^0(M,\tau)$ almost uniformly $($\emph{a.u} in short$)$ $($respectively, bilaterally almost uniformly $($ \emph{b.a.u} in short$))$ if for every $\epsilon>0$, there exists a projection $e\in M$ such that 
	\begin{align*}
	\tau(1-e)&\leq \epsilon \quad \text{ and }\quad \lim_{n\rightarrow\infty} \norm{(x_n-x) e}=0\\
	(\text{respectively, }\tau(1-e)&\leq \epsilon \quad \text{ and }\quad \lim_{n\rightarrow\infty} \norm{e(x_n-x) e}=0).
	\end{align*}
\end{defn}
We note that if $(x_n)$ converges in \emph{a.u}, then it converges in \emph{b.a.u}. The following results  are standard fact of the subject. For completeness,  we add their proofs.

\begin{lem}\label{dense}
	$ M \cap L^1(M, \tau)$ is dense in $L^\Phi(M, \tau)$. Moreover, for $x \in L^\Phi(M, \tau)_+$, then there exists an increasing  sequence $(x_k) \subseteq   M \cap L^1(M, \tau)_+$  such that $\underset{ k \ri \infty }{\lim} \norm{ x - x_k }_\Phi = 0 $   and satisfying the following; 
	$$  x = \text{a.u-}\underset{ n \ri \infty }{\lim}  x_k .$$
	%\begin{enumerate}
	%	\item $ \underset{ k \ri \infty }{\lim}  \tau( y x_k) = \tau( yx)$ and 
	%	\item  $ x = \text{b.a.u-}\underset{ n \ri \infty }{\lim}  x_k $.
	%\end{enumerate} 
\end{lem}
\begin{proof}
Let $ x \in  L^\Phi(M, \tau)_+$ and  consider $ e_n = \chi_{ ( 0, n ]}(x)$.
Suppose $ x_n= xe_n $, then clearly $ x_n \in M$ and $\int^n_0  \lambda d e_\lambda $, where $\{ e_\lambda  : \lambda \geq 0 \}$ is the set of spectral projections of $x$. . Since $ x$ is measurable, we note that $ \underset{ n \ri \infty }{\lim} \tau( 1-e_n) = 0$ (see \cite[Proposition 4.7]{Hi21}).

Further as $ (x-xe_n)e_n=0$  for all $ n \in \N$, it follows that 
the sequence  $ (x-xe_n) $ converges to $0$ in measure topology. Consequently, by using 
\cite[Lemma 4.18 (4)]{Hi21}, we obtain that 
\begin{align}\label{measure}
\lim_{ n \ri \infty } \mu_t( x-xe_n) \ri 0  \text{ for all   } t > 0.
\end{align}
Further, we have $ 0\leq x -xe_n \leq x $ for all $n \in \N$, so, from \cite[Proposition 4.19 (5)]{Hi21}, it follows that 
$$ \mu_t(  x -xe_n ) \leq \mu_t(x) \text{ for all } t >0 \text{ and } n \in \N.$$
Now we apply  Lebesgue dominated convergence theorem (see \cite[Proposition 4.25]{Hi21}), we obtain the following 
$$ \lim_{ n \ri \infty } \tau( \Phi( x-xe_n)) = \lim_{ n \ri \infty } \int_0^\infty \Phi(\mu_t( x-xe_n)) dt = 0.$$
Therefore, $\underset{ n \ri \infty }{ \lim }  \norm{ x-xe_n}_\Phi = 0$, i.e, $\underset{ n \ri \infty }{ \lim }  \norm{ x-x_n}_\Phi = 0$ .  As  $ x_n \in M \cap L^\Phi(M, \tau)$, so,  $ M \cap L^\Phi(M, \tau)$ is dense in $L^\Phi(M, \tau)$. 

To prove  $ M \cap L^1(M, \tau))$ is dense in $L^\Phi(M, \tau)$. Let $ x \in L^\Phi(M, \tau)_+$, then  consider,   $ p_n =   \chi_{ ( \frac{1}{n}, n ]}(x)$ and  $ y_n = xp_n $. Then  note that $ (y_n) \subseteq  M \cap  L^1(M, \tau)$. Further, suppose $ q_n =   \chi_{ [ 0, \frac{1}{n} ]}(x)$. It implies  that $ \norm{ xq_n } \leq \frac{1}{n}$ and hence, $ ( xq_n)$ converges to $0$ in norm. Consequently, $ (\mu_t( xq_n))$ converges to $0$ for all $ t >0$.  Hence, $( \norm{( xq_n)}_\Phi)$ converges to $0$. Now observe that 
\begin{align*}
\norm{ x-y_n}_\Phi  &\leq  \norm{ x  \int_0^{ \frac{1}{n}} \lambda d e_\lambda )}_\Phi + \norm{ x  \int_n^{ \infty} \lambda d e_\lambda )}_\Phi\\
& \leq  \norm{ xq_n}_\Phi  + \norm{ x e_n^\perp }_\Phi \ri 0, \text{ as } n \ri \infty.   
\end{align*}

Note that $(x-y_n)$ converges to $0$ in measure topology and $ y_n \leq x $ for all $n \in \N$. So, we can find a subsequence $(y_{ n_k})$ which converges to $x$  in \emph{a.u}.  This completes the proofs.
\end{proof}

\begin{defn}
A linear operator $ T : M + L^1(M, \tau) \rightarrow M + L^1(M, \tau)$ is called Dunford-Schwartz 
operator if it satisfies the following;
\begin{enumerate}
	\item $\norm{T(x)}\leq \norm{x} $ for all $x \in M$ and 
	\item $\norm{ T(x)}_1 \leq \norm{x}_1 $ for all $ x \in L^1(M, \tau)$. 
\end{enumerate}
\end{defn}
The following result is a known fact in the literature, but for plain reading we add a proof. It will be used in this article.
\begin{prop}\label{ext}
Let $M$ be a semifinite von Neumann algebra with  a f.n semifinite trace $\tau$. Suppose $T: M \rightarrow M$  is a linear having the following property 
\begin{enumerate}
	\item contraction: $ \norm{T(x)} \leq \norm{x} $ for all $ x \in M$,
	\item positivity:  $ T(x) \geq 0 $ for all $ x \in M_+$ and 
	\item subtracial : $\tau \circ T \leq \tau$ i.e, $ \tau(T(x)) \leq \tau(x)$ for all $ x \in M \cap L^1(M, \tau)_+$. 
\end{enumerate}
Then $T$ extends to a positive contraction on   $L^\Phi( M, \tau)$. 
\end{prop}
\begin{proof}
We readily note that $T$  extends to $ L^1(M, \tau)_+$ and it satisfies $ \norm{( T(x))}_1 \leq \norm{x}_1$ for all $ x \in L^1(M, \tau)_+$. Then we note that $T$ extends to a bounded positive map on $ L^1(M, \tau)$ and still denoted by $T$. Indeed, let $ x \in L^1(M, \tau)$ such that 
$ x = x^*$ and suppose $ x = x_+ -x_-$ where $ x_+, x_-\geq 0$ and $x_+ x_- = 0 $. Then note that 
$$\norm{T(x)}_1 = \tau(T(x_+) ) + \tau(T(x_-) ) \leq \tau(x_+) + \tau(x_1) = \norm{x}_1$$
i.e, $ \norm{T(x)}_1 \leq \norm{x}_1 $ for all $ x \in L^1(M, \tau)$ with $ x = x^*$. 
Now suppose $ x \in L^1(M, \tau)$ and write $ x = x_1-x_2 + i ( x_3-x_4)$ where 
$ x_j \in L^1(M, \tau)_+$ for $ j = 1, 2, \cdots, 4$. Further, we also have  $\norm{ x_1-x_2}_1 \leq \norm{ x}_1$ and$\norm{ x_3-x_4}_1 \leq \norm{ x}_1$. Then we conclude that 
$$ \norm{T(x)}_1 \leq 2\norm{ x }_1, \text{ for all } x \in L^1(M, \tau).$$
Let $ S : M \ri M$ be be the dual map of $T$ and  it satisfies the following 
$$ \tau( S(x) y ) = \tau( x T(y)) \text{ for } x \in M \text{ and } y \in L^1(M, \tau). $$
Note that $S$ is a bounded positive map  and hence $ \norm{ S} = \norm{S(1)} \leq  1$ (see \cite[Theorem 6.9]{Pa02}. Therefore, 
$\norm{ T} \leq 1$. Since  $ L^\Phi( M , \tau)$ is the exact interpolation space for the Banach couple $ ( L^1(M, \tau), M )$ (see \cite[Theorem 3.4]{DD92}), so, by complex interpolation $T$ extends to a positive contraction on $L^\Phi(M, \tau)$.

Indeed, let $ x \in L^\Phi(M, \tau) $ and suppose $ x = y+ z $ for some $ y, z \in L^0(M, \tau)$,  then as  $T$ is Danford-Schwartz, so, for all $ t >0$, we  note that 
\begin{align*}
\tau( \abs{T(y)}) + t \norm{T(z)} & = \norm{ T(y)}_1 + t \norm{T(z)}\\
&  \leq \norm{ y}_1 + t \norm{z} = \tau( \abs{y}) + t \norm{z}.
\end{align*}
Therefore,  from Proposition \ref{Prop:IntegralDecRea}, we deduce the following: 
$$ \int_0^t \mu_s(T(x))\,ds \leq \int_0^t \mu_s(x)\,ds, \text{ for all } x \in L^\Phi(M, \tau) \text{ and } t >0.$$
Then use the Hardy-Littlewood-Polya inequality (Theorem \ref{Thm:HLP}) to obtain the following:
	\begin{align*}
\int_{0}^t \Phi(\mu_s(T(x)))\, ds\leq \int_0^t \Phi(\mu_s(x))\, ds, \text{ for all } x \in L^\Phi(M, \tau) \text{ and } t >0.
\end{align*}
Hence, we obtain 
$$ \norm{ T(x) }_\Phi \leq \norm{ x}_\Phi, \text{ for all } x \in  L^\Phi(M, \tau).$$

\end{proof}

\subsection{Vector-Valued $L^p$-spaces.} In this section, we discuss the noncommutative vector valued  $L^p$-spaces.  Given $1\leq p\leq \infty$,  the noncommutative $L^p(M,\ell^\infty)$ is defined as the collection of all sequences $x=(x_n)_{n\geq 0}$ in $L^p(M, \tau)$ for which there exist  $a,b\in L^{2p}(M, \tau)$  and  $y=(y_n)_{n\geq 0}\subseteq L^\infty(M, \tau)$ such that $x_n=ay_n b$  for $n\geq 0$. The space $L^p(M,\ell^\infty)$, equipped with the norm
\begin{align*}
    \norm{x}_{L^p(M,\ell^\infty)}=\inf\{\norm{a}_{2p}\sup_{n\geq 0}\norm{y}_\infty \norm{b}_{2p}\},
\end{align*}
is a Banach space, where the infimum runs over all factorizations as above. It is known that for a positive sequence $x=(x_n)\in L^p(M,\ell^\infty)$ if and only there is an element $a\in L^p(M,  \tau)_+$ such that $x_n\leq a$ for all $n\geq 0$. Moreover, we have
\begin{align*}
    \norm{x}_{L^p(M,\ell^\infty)}=\inf\{\norm{a}_{p}: a\in L^p(M)_+ \text{ such that }x_n\leq a,\,\,\forall n\geq 0\}.
\end{align*}
We sometimes write $\norm{\sup_n^+ x_n}_p$ for $\norm{x}_{L^p(M,\ell^\infty)}$, and these two notations will be used interchangeably.

%%%%%%%%
%%%%%%%%%%%%%%%%%%%%%%%%%%%%%%%%%%%%%%%%%%%%%%%%%%%%%%%%%%%%%%%%%%%%%%%%%%%%%%%%%%%%%%%%%%%%%%%%%%%%%%%%%%%%%%%%%%%%%%%%%%%%%%vectororlicz space no nee %%%%%%%%%%%%%%%%%%%%%%%%%%%%%%%%%%%%%%%%%%%%%%%%%%%%%%%%%%%%%%%%%%%%%%%%%%%%%%%%%%%%%%%%%%%%%%%%%%%%%%%%%%%%%%%%%%%%%%%%%%%%%%%%%%%%%%%%%%%%%%%%%%%%%%%%%%%%%%%%%%%%%%%%%%%%%%%%%%%%

%%%%%%%%
%%%%%%%%%%%%%%%%%%%%%%%%%%%%%%%%%%%%%%%%%%%%%%%%%%%%%%%%%%%%%%%%%%%%%%%%%%%%%%%%%%%%%%%%%%%%%%%%%%%%%%%%%%%%%%%%%%%%%%%%%%%%%%%%%%%%%%%%%%%%%%%%%%%%%%%%%%%%%%%%%%%%%%%%%%%%%%%%%%%%%%%%%%%%%%%%%%%%%%%%%%%%%%%%%%%%%%%%%%%%%%%%%%%%%%%%%%%%%%%%%%%%%%%%%%%%%%%%%%%%%%%%%%%%%%%%%%%%%%%%%%%%%%%%%%%%%%%%%%%%%%%%%%%%%%%%

\subsection{Some technical results regarding \emph{b.a.u} and \emph{a.u} convergence}
In this subsection, we begin with  recalling  the definitions of \emph{b.u.e.m} and \emph{u.e.m}. Then  discuss  their relations with  \emph{b.a.u} and \emph{a.u} convergence.  In the ergodic theory to established an ergodic theorem, one of the main technique is to establish an appropriate maximal inequality. But, often it becomes very technical to established a maximal inequality. However,  in this article, we use 
\emph{b.u.e.m }  and \emph{u.e.m} continuity  to establish  some of our important 
ergodic theorems.\\
These materials are already available in the literature. We are not claiming its originality. However, for the sake of completeness of this paper and plain reference and reading, we will present some of the results along with their proofs.

In this subsection we consider a norm space  $(X,\norm{\cdot})$  and  each $n \in \N$,   $A_n: X \rightarrow L^0(M, \tau)$ denotes a additive map. We begin with  recalling the following definitions.

% and $e\in\mathcal{P}(M) $ be such that $\{A_n(x)e:n\in \mathbb{N}\}\subseteq M $ $($respectively, $\{e A_n(x)e:n\in \mathbb{N}\}\subseteq M )$. Following \cite{Li12}, we define
%\begin{align*}
%	A^*(x,e)=\sup_n \norm{A_n(x)e}_\infty \quad(\text{respectively, }A_b^*(x,e)=\sup_n %\norm{eA_n(x)e}_\infty).
%\end{align*}

%===================================%

\begin{defn}\label{Defn:BUEM}
	Let $(X,\norm{\cdot})$	be a normed space and $X_0$ be a subset of $X$ such that $0\in X_0$. For each $n \in \N$, consider a additive map 
	$A_n:X\rightarrow L^0(M, \tau)$. Then we have the following definitions. 
    \begin{enumerate}
		\item The sequence $(A_n)$ is called \textit{uniformly	equicontinuous in measure} $($\textit{u.e.m.} in short$)$ at $0$ on $X_0$, if given $\epsilon, \delta > 0$, there is a $\gamma > 0$ such that for every $x\in X_0$ with $\norm{x} < \gamma$, there exists a projection $e:=e(x)\in \mathcal{P} (M )$ for which
    	\begin{align*}
		\tau(e^\perp)<\epsilon\text{ and } \underset{n \in \N}{\sup} \norm{A_n(x)e}\leq \delta \text{ hold}.
	    \end{align*}
	
	   \item  The sequence $(A_n)$ is called \textit{bilaterally uniformly equicontinuous in measure} $($\textit{b.u.e.m.} in short$)$ at $0$ on $X_0$ if, given $\epsilon, \delta > 0$, there is a $\gamma > 0$ such that for every $x\in X_0$ with $\norm{x}_X < \gamma$, there exists a projection $e:=e(x)\in \mathcal{P} (M )$ for which
	   \begin{align*}
	    \tau(e^\perp)<\epsilon\text{ and } \underset{n \in \N}{\sup} \norm{eA_n(x)e}\leq \delta \text{ hold}.
	    \end{align*}
	
	  \item  The sequence $(A_n)$ is called \textit{pointwise uniformly bounded}  $($\textit{p.u.b.} in short$)$ on $X_0$, if for all $ x \in X_0$ and for $\epsilon > 0$, there exists a projection $e:=e(x)\in \mathcal{P} (M )$ such that 
	  \begin{align*}
		\tau(e^\perp ) \leq \eps \text{ and } \sup_{ n \geq 0} \underset{n \in \N}{\sup} \norm{A_n(x)e} < \infty.
	  \end{align*}
	
      \item   The sequence $(A_n)$ is called \textit{pointwise bilateral uniform bounded} $($\textit{p.b.u.b.} in short$)$  if for all $ x \in L^p(M)$ and for $\epsilon > 0$, there exists a projection $e:=e(x)\in \mathcal{P} (M )$ such that 
	  \begin{align*}
		\tau(e^\perp ) \leq \eps \text{ and } \underset{n \in \N}{\sup} \norm{eA_n(x)e}  < \infty.
   	  \end{align*}
  \end{enumerate}
	
\end{defn}

%===================================%
The following lemma is possibly well known but we could not locate a proper reference. The proof is similar to the proof of \cite[Lemma 4.1]{Li12}. Thus, we skip the proof.

\begin{lem}\label{Lem:BUEMonPositive}
Let $A_n: L^\Phi(M)\rightarrow L^0(M, \tau)$ be a sequence of additive positive maps. Then the family $\{A_n\}$ is $u.e.m.$ $($respectively, \textit{b.u.e.m.}$)$ at $0$ on $(L^\Phi(M),\norm{\cdot}_\Phi)$ if and only if $\{A_n\}$ is $u.e.m.$ $($respectively, \textit{b.u.e.m.}$)$  at $0$ on $(L^\Phi(M)_+,\norm{\cdot}_\Phi)$.
\end{lem}

\begin{proof}
	Suppose  $(A_n)$ is \emph{ b.u.e.m} at $0$ on $L^\Phi(M, \tau)_+$.
	Let $ \eps >0$ and $ \delta $, then there exists a $ \gamma > 0$ such that 
	for all $ x \in L^p(M, \tau)_+ $ with $ \norm{x}_\Phi < \gamma $, then there exists $ e_x \in \CP(M)$ such that 
	$$ \tau(e_x^\perp ) \leq \frac{\eps}{4} \text{ and  }  \sup_{ n  } \norm{ e_x S_n(x)e_x} \leq \frac{\delta}{4} $$
	
	Now let $ x \in L^\Phi(M, \tau)$ such that $ \norm{x}_\Phi < \gamma $. Now write 
	$ x = x_1 -x_2 + i ( x_3 -x_4)$ where $ x_j \in L^\Phi(M, \tau)_+$,  and $\norm{x_j}_\Phi \leq \norm{x}_\Phi < \gamma $ for all $ 1\leq j \leq 4 $.  
	Then exists $ e_j \in \CP(M)   $ such that $ \tau(e_j^\perp)  \leq \frac{ \eps}{4}$ and $ \norm{ e_j A_n(x_j) e_j} \leq \frac{\delta}{  4}$ for $ 1 \leq j \leq 4 $.  Now suppose $ p_x = \displaystyle \bigwedge_{j = 1}^4 e_j $. Then we have $ \tau(p_x^\perp ) \leq \eps $ and  observe that 
	\begin{align*}
	\norm{ p_x A_n(x)p_x} = \norm{ p_x A_n((x_1-x_2) +i ( x_3-x_4))p_x}\\
	\leq \sum_{j = 1}^4 \norm{ p_x A_n(x_j)p_x}
	\leq \sum_{j = 1}^4 \norm{ e_j A_n(x_j)e_j} \leq 4 \cdot \frac{\delta}{4}= \delta,
	\end{align*}
	Thus this complete the proof.
\end{proof}

\begin{thm}\cite[Theorem 3.1]{Li12}\label{Thm:BUEMonClosure}
	Let $(X,\norm{\cdot})$ be a Banach space, and let $A_n:X\rightarrow L^0(M, \tau)$ be a sequence of continuous additive maps. If the family $\{A_n\}$ is $b.u.e.m.$ at $0$ on a set $X_0\subset X$ such that $\{A_n(x)\}\subseteq L^0(M, \tau)_+$ for all $x\in \overline{X_0}$, where $ \overline{X_0}$ is the closure of $X_0$ in $X$, then it is also $b.u.e.m.$ at $0$ on $ \bar{X_0}$.
\end{thm}

Let us recall the following lemma from \cite{GL00}.

\begin{lem}\cite[Lemma 3]{GL00}\label{Lem:AUBAU}
	Let $b\in M$ with $0\leq b\leq 1$. If $e$ is the spectral projection of $b$ corresponding to the interval $[1/2,1]$, then the following hold:\\
	\begin{enumerate}\addtolength{\itemsep}{0.15cm}
		\item $\tau(e^\perp)\leq 2\tau(1-b)$;
		\item there exists $b^-\in M$ such that $\norm{b^-}_\infty\leq 2$ and $e=bb^-$.
	\end{enumerate}
\end{lem}

As b.u.e.m continuity is an alternative approach to establish  ergodic theorems other than proving a maximal inequality, the 
 following result is  important to established the b.u.e.m continuity.  It implicitly appears in many places, for e.g., \cite{GL00}, \cite{LM01}, \cite{CLS05} and  \cite{HS08} and we add a proof for plain reading. 

\begin{thm}\label{Thm:BoundedMeasure} 
Let $A_n:  L^\Phi(M, \tau) \rightarrow L^0(M, \tau)$ be a sequence of continuous additive maps such that $A_n$ is \textit{p.b.u.b.}$($respectively, \textit{p.u.b}$)$ on $L^\Phi(M, \tau)$. Then $A_n$ is \textit{u.e.m.} 	$($respectively, \textit{b.u.e.m.}$)$ at $0$ on $L^\Phi(M, \tau)$.
\end{thm}
\begin{proof}
We prove this result only for 	\textit{b.u.e.m.}. The proof of the other part is similar.
Fix $ \epsilon > 0$  and  $\delta > 0$. For $ k \in \mathbb{N}$, we define 
\begin{align*}
E_k = \{ x\in L^\Phi(M) :  \sup_{n \in \mathbb{N}}&\norm{ A_n(x)^{1/2} b  }_\infty \leq k, \text{ for some } b \in M\\
 &\text{ with } 0< b \leq 1 \text{ and } \tau(1-b) \leq \frac{\epsilon}{4}  \}.
\end{align*}
In two steps, we prove that 
	\begin{enumerate}
		\item $E_k$ is closed and
		\item $ L^\Phi(M)_+ = \cup_{ k \in \N }E_k$.
	\end{enumerate}

At first  we prove $(1)$, i.e., we show that $E_k$ is closed for each $k\in\mathbb{N}$.

Let $ ( a_m)$ be a sequence in $E_k$ such that $(a_m)$ converges to $a \in L^\Phi(M)$ in $\norm{\cdot}_\Phi$. 
	We want to show that $ a \in E_k$. Since $A_n :L^\Phi(M)\rightarrow L^0(M, \tau)$ is continuous, for $n \in \mathbb{N}$, $A_n(a_m)$ converges to $A_n(a)$ in the measure topology. Then for each $n\in \N$, there exists a subsequence	$(a_{n,m})_{m \in \mathbb{N}}$  of $(a_m)$ such that $A_n( a_{n, m})^{1/2} $ converges to $A_n(a)^{1/2} $ in a.u.\\
		Set $ x_m = a_{m, m}$, $m\in \mathbb{N}$. Since $( x_m)_{m \geq n}$ is a subsequence of $ ( a_{n, m})_{ m \geq 1} $, we have 
	\begin{align*}
	A_n(x_m)^{1/2} \xrightarrow{ m \rightarrow \infty } A_n(a)^{1/2} \text{ in a.u.}, \text{ for all } n \geq 1.
	\end{align*}  
 Since  for each $m \in \mathbb{N}$, $ x_m \in E_k $, there exists $ b_m \in M $ such that 
	\begin{enumerate}
		\item $0 < b_m \leq 1 $ with $ \tau(1-b_m) \leq {\epsilon/4}$, and  
		\item $\displaystyle \sup_{n \in \N} \norm{ S_n(x_m)^{1/2} b_m }_\infty \leq k$. 
	\end{enumerate}

Now since the unit ball $M_1$ of $M$ is weak operator topology compact, by dropping to a subsequence if necessary, one can assume that $(b_m)$ converges to an element $b \in M_1$ in  weak operator topology. Now we will show that 
	\begin{enumerate}
		\item[(3)] $ 0 < b \leq 1  $ with $ \tau(1-b) \leq \epsilon/4$ and 
		\item[(4)] $\displaystyle \sup_{n \in \N} \norm{ A_n(a)^{1/2} b }_\infty \leq k$.
	\end{enumerate}
	Note that $(3)$ is obvious as $ b= \text{ WOT-}\lim b_m$ and hence, $\tau(1-b) \leq \liminf \tau(1-b_m) \leq  \epsilon/4$. 
For $(4)$, we fix a $ n \in \mathbb{N}$. Since $ A_n(x_m)^{1/2} $ converges to $ A_n(a)^{1/2}$ in a.u., given $\gamma > 0$,  there exists a projection $ e \in \mathcal{P}(M)$ such that 
\begin{align*}
\norm{e\Big( A_n(x_m)^{1/2} - A_n(a)^{1/2} \Big)}_{\infty} \xrightarrow{ m \rightarrow \infty } 0 \text{ and } \tau( 1-e) \leq \gamma.
\end{align*}
Now we show that 
\begin{align*}
\norm{e A_n(a) b }_\infty \leq k.
\end{align*}

Let $ \xi, \eta \in L^2(M)$ be such that $ \norm{\xi}_2 = \norm{\eta}_2 = 1$. For any $\delta^\prime > 0$, there exists $m_0\in \mathbb{N}$ such that 
\begin{align*}
\abs{\left\langle (b_{m}- b) \xi,  A_n(a)^{1/2} e \eta\right\rangle}< \delta^\prime, \text{for all } m \geq m_0.
\end{align*}
Now observe that
\begin{align*}
\abs{\left\langle e A_n(a)^{1/2} b \xi, \eta\right\rangle}
&\leq \abs{\left\langle e A_n(a)^{1/2} (b-b_{m} ) \xi, \eta\right\rangle} +\abs{\left\langle e A_n(a)^{1/2} b_m \xi,\, \eta\right\rangle}\\
&\leq \delta^\prime+  \abs{\left\langle  e \Big(A_n(a)^{1/2} -A_n(x_m)^{1/2} +A_n(x_m)^{1/2} \Big)  b_{m}  \xi,\,\eta \right\rangle}\\
&\leq \delta^\prime+ \norm{e \Big(A_n(a)^{1/2} -A_n(x_m)^{1/2} \Big)b_m} +\abs{\left\langle A_n(x_m)^{1/2}  b_{m}  \xi,\,\eta \right\rangle}\\
&\leq  \delta^\prime + k + \norm{e\Big( A_n(x_m)^{1/2} - S_n(a)^{1/2} \Big)}_\infty\\
 &\xrightarrow{m \rightarrow \infty} \delta^\prime +k.
\end{align*}
Thus,  
\begin{align*}
 \norm{e A_n(a)^{1/2} b }_\infty \leq  \delta^\prime +k.
\end{align*}
But $ \delta^\prime> 0$ is arbitrary. So, $ \norm{e A_n(a)^{1/2} b } \leq k $ with $ \tau(1- e)  < \gamma $.  Again, since $\gamma>0$ is arbitrary, we have 
\begin{align*}
	 \norm{A_n(a)^{1/2} b } \leq k .
\end{align*}
This shows that $E_k$ is closed. \\
Now we show that $ L^\Phi(M)_+ = \cup_{ k \in \mathbb{N }} E_k $.
Indeed, for every $x \in L^\Phi(M)_+ $ and  given $\epsilon > 0$, there exists a projection $e \in \mathcal{P}(M)$ such that 
\begin{align*}
\tau(e^\perp) < \epsilon/4  \text{ and } \sup_{n \in \mathbb{N}} \norm{ e A_n(x)e }_\infty < \infty.
\end{align*}
Observe that
\begin{align*}
\norm{ A_n(x)^{1/2} e }_\infty^2
= \norm{ e A_n(x)^{1/2} A_n(x)^{1/2} e }_\infty
= \norm{ e A_n(x)e }_\infty
\leq \sup_{n \in \N} \norm{ e A_n(x)e }_\infty.
\end{align*}
Therefore, $ x  \in E_k $ for  every $ k$ such that $k \geq \displaystyle \sup_{n \in \N} \norm{ e A_n(x)e }_\infty$.

Now we like to conclude that  that  $A_n$ is \textit{b.u.e.m.} at $0$ on $L^\Phi(M)$.
Notice that $ L^\Phi(M)_+ $ is complete. By Baire category theorem, there exists a open set $U$ and $k_0\in\mathbb{N}$ such that 	
\begin{align*}
U \cap L^\Phi(M)_+  \subseteq E_{k_0}.
\end{align*}
That means, given $\epsilon >0 $, and for each $ x \in  U \cap L^\Phi(M)_+ $,  there exists an element $b_x \in M$  satisfying the following conditions:
	\begin{enumerate}
		\item[(iii)] $ 0< b_x \leq 1$  with $\tau(1-b_x) \leq \epsilon/4$,  and 
		\item[(iv)] $\displaystyle \sup_{n \in \mathbb{N}} \norm{ A_n(x)^{1/2} b_x }_\infty \leq k_0$.
	\end{enumerate}

Set $ e_x = 1_{ [1/2 , 1]}(b_x)$. By Lemma \ref{Lem:AUBAU}, we have 
\begin{enumerate}
\item[(v)] 	$ \tau(e_x^\perp ) \leq \epsilon/2$, and
\item[(vi)] $\sup_{n \in \N}\norm{ e_x A_n(x) e_x}_\infty\leq 4 k^2_0$.
\end{enumerate}
Indeed, 
\begin{align*}
	\sup_{n \in \N}\norm{ e_x A_n(x) e_x}_\infty &=  \sup_{n \in \N}\norm{ A_n(x)^{1/2} e_x}_\infty^2 \\
	&\leq \sup_{n \in \N}\norm{ 2 A_n(x)^{1/2} b_x}_\infty^2 \\
	&=4\sup_{n \in \N}\norm{  S_n(x)^{1/2} b_x}_\infty^2 \\
	&\leq 4 k^2_0.
\end{align*}

Now we show that $A_n$ is \textit{b.u.e.m} at $0$ on $L^\Phi(M)$. Let $\delta >0$ and $ y \in U $. 
Then one can find an open set $V$ at $0$ such that $y + V \subseteq U$. Let $t $ be a positive integer such that $\frac{8k_0^2}{t} \leq \delta $.  Then one can further find an open set $W$ at $0$ such that $ t W \subseteq V$.

Suppose $ x \in W \cap L^\Phi(M)_+ $. Set $ z = y + t x $. Then note that $ z \in U$. Let $e_z, e_y\in\mathcal{P}(M)$ be satisfying the conditions $(v)$ and $(vi)$. Set $ q = e_z \wedge e_y$. Note that $\tau(q^\perp)\leq \epsilon$. Moreover, we have 
	\begin{align*}
		\norm{ q A_n(x)q}_\infty &
		= \frac{1}{t}  \norm{ q A_n(tx)q}_\infty\\
		& =\frac{1}{t}  \norm{ q A_n(z-y)q}_\infty\\
		&\leq  \frac{1}{t}  \norm{ q A_n(z)q}_\infty + \frac{1}{t}  \norm{ q A_n(y)q}_\infty\\
		&\leq  \frac{1}{t}  \norm{ e_z A_n(z)e_z}_\infty + \frac{1}{t}  \norm{ e_y A_n(y)e_y}_\infty\\
		&	\leq \frac{8k_0^2}{t}\\
		& \leq  \delta.
	\end{align*}
This shows that $A_n$ is \textit{b.u.e.m} at $0$ on $L^\Phi(M)$. This completes the proof.
\end{proof}

\begin{defn}\label{Defn:AUBAUconvergence}
	Let $M$ be a von Neumann algebra equipped with a $f.n.s.$ trace $\tau$.  A sequence $(x_n)$ in $L^0(M,\tau)$ is said to be converges to $x\in L^0(M,\tau)$ almost uniformly $($\emph{a.u} in short$)$ $($respectively, bilaterally almost uniformly $($ \emph{b.a.u} in short$))$ if for every $\epsilon>0$, there exists a projection $e\in M$ such that 
	\begin{align*}
	\tau(1-e)&\leq \epsilon \quad \text{ and }\quad \lim_{n\rightarrow\infty} \norm{(x_n-x) e}=0\\
	(\text{respectively, }\tau(1-e)&\leq \epsilon \quad \text{ and }\quad \lim_{n\rightarrow\infty} \norm{e(x_n-x) e}=0).
	\end{align*}
\end{defn}

The following is a noncommutative variant of the Banach Principle for \emph{b.a.u} $($respectively,  \emph{a.u}$)$ convergence. This appeared many places in different context. We refer \cite[Theorem 1]{LM01}, \cite[Theorem 2.7]{CLS05} for b.a.u. convergence, \cite[Theorem 2 and Remark 1]{GL00} for a.u. convergence, or \cite[Theorem 2.3]{CL06}.

\begin{thm}\label{Thm:BanachPrincipleAu}
Let $(X, \norm{\cdot})$ be a normed space and $A_n : X \rightarrow L^0(M, \tau)$ be a sequence of additive maps. If the family 
$\{A_n\}$ is \textit{u.e.m.} $($respectively, \textit{b.u.e.m.}$)$ at $0$ on $X$, then the set 
\begin{align*}
C=\{x\in X:  (A_n(x)) \text{ converges in \emph{b.a.u} } (\text{respectively, in \emph{ a.u.}})\},
\end{align*}
is closed in $X$.
\end{thm}

\section{Noncommutative  Rota's Theorem}\label{Rota}
In this section, we revisit the actions of a free group and prove the Rota ``Alternierende Verfahren" theorem on noncommutative Orlicz space. 
 Our approach involves the noncommutative Kolmogorov construction/noncommutative Markov chain construction by Anantharaman-Delaroche \cite{Ad06} and a result by Litvinov \cite{Li12}. First, let's recall the noncommutative Kolmogorov construction by Anantharaman-Delaroche \cite{Ad06}.

%%%%%%%%%%%%%%%%%%%%%%%%%%%%%%%%%%%%%%%%%%%%%%%%%%%%%%%%%%%%%%%%%%%%%%%%%%%%%%%%%%%%%%%%%%%%%%%%%%%%%%%%%%%%%%%%%%%%%%%%%%%%%%%%%%%%%%%%%%%%%%%%%%%%%%%%%%%%%%%%%%%%%%%%%%%%%%%%%%%%%%%%%%%%%%%%%%%%%%%%%%%%%%%%%%%%%%%%%%%%%%%%%%%%%%%%%%%%%%%%%%%%%%%%%%%%%%%%%%%%%%%%%%%%%%%%%%%%%%%%%%%%%%%%%%%%%%%%%%%%%%%%%%%%%%%%%%%%%%%%%%%%%%%%%%%%%%%%%%%%%%%%%%%%%%%%%%%%%%%%%%%%%%%%%%%%%%%%%%%%%%%%%%%%%%%%%%%%

%%%%%%%%%%%%%%%%%%%%%%%%%%%%%%%%%%%%%%%%%%%%%%%%%%%%%%%%%%%%%%%%%%%%%%%%%%%%%%%%%%%%%%%%%%%%%%%%%%%%%%%%%%%%%%%%%%%%%%%%%%%%%%%%%%%%%%%%%%%%%%%%%%%%%%%%%%%%%%%%%%%%%%%%%%%%%%%%%%%%%%%%%%%%%%%%%%%%%%%%%%%%%%%%%%%%%%%%%%%%%%%%%%%%%%%%%%%%%%

 %%%%%%%%%%%%%%%%%%%%%%%%%%%%%%%%%%%%%%%%%%%%%%%%%%%%%%%%%%%%%%%%%%%%%%%%%%%%%%%%%%%%%%%%%%%%%%%%%%%%%%%%%%%%%%%%%%%%%%%%%%%%%%%%%%%%%%%%%%%%%%%%%%%%%%%%%%%%%%%%%%%%%%%%%%%%%%%%%%%%%%%%%%%%%%%%%%%%%%%%%%%%%%%%%%%%%%%%%%%%%%%%%%%%%%%%%%%%%%%%%%%%%%%%%%%%%%%%%%%%%%%%%%%%%%%%%%%%%%%%%%%%%%%%%%%%%%%%%%%%%%%%%%%%%%%%%%%%%%%%%%%%%%%%%%%%%%%%%%%%%%%%%%%%%%%%%%%%%%%%%%%%%%%%%%%%%%%%%%%%%%%%%%%%%%%%%%%%%%%%%%%%%%%%%%%%%%%%%%%%%%%%%%%%%%%%%%%%%%%%%%%%%%%%%%%%

\begin{defn}\label{facto}
	Let $(M, \phi)$ and $(N, \psi)$ be two von Neumann probability space  and $T: M\ri N$ be $( \phi, \psi)$-preserving map Markov operator. Then $T$ is called \emph{ factorable } if there exists
	a von Neumann probability space $ ( R, \omega )$ and two normal unital homomorphisms
	\begin{enumerate}
		\item $ \alpha: M \ri R$ which is $( \phi, \omega)$-preserving  and 
		\item $ \beta: R \ri N$ which is $( \omega, \psi)$-preserving  
	\end{enumerate}
	such that $ T = \beta \circ \alpha$. 
	
\end{defn}
\begin{thm}[\cite{Ad06},\textbf{ Noncommutative Markov chain construction} ]\label{Thm:NCMarkovConst}
	Let $(M, \varphi)$ be a noncommutative probability measure space.  Let $P:(M,\varphi)\rightarrow (M,\varphi)$ be a factorizable $\varphi$-preserving Markov operator. Then there exists 
	\begin{enumerate}

		\item a von Neumann algebra $B$,
		\item a normal faithful state $\rho$ on $B$,
		\item a normal unital endomorphism $\beta:B\rightarrow B$ with $\rho\circ\beta=\rho$ and $\sigma_t^\rho\circ\beta=\beta\circ\sigma_t^\rho$ for all $t\in\mathbb{R}$
		\item  a normal unital homomorphism $J_0:B\rightarrow B$ with $\rho\circ J_0=\psi$ and $\sigma_t^\rho\circ J_0=J_0\circ\sigma_t^\psi$ for $t\in\mathbb{R}$,
	\end{enumerate}
	such that, if we set $J_n=\beta^n \circ J_0$ for $n\geq 0$, and if $B_{n]}$ and $B_{[n}$ denotes the von Neumann subalgebras of ${B}$ generated by $\cup_{k\leq n} J_k(B)$ and $\cup_{k\geq n} J_k(B)$ respectively, then
	\begin{enumerate}
		\item the algebras ${B}_{n]}$ and ${B}_{[n}$ are $\rho$ invariant;
		\item if $\mathbb{E}_{n]}$ and $\mathbb{E}_{[n}$ are the corresponding $\phi$-preserving conditional expectations, for $n\in\mathbb{N}$ and $q\geq n$, we have
	\end{enumerate}
	\begin{align}
		\mathbb{E}_{n]}\circ J_q &=J_n\circ P^{q-n},\\
		\mathbb{E}_{n+q]}\circ \beta^q &=\beta^q\circ \mathbb{E}_{n]},\\
		\mathbb{E}_{[n}\circ J_0 &=J_n\circ (P^*)^n.
	\end{align}
\end{thm}

\begin{thm}[\cite{Ad06}]\label{Thm:FactoMap}
Let $T$ be a factorable operator on $(M,\varphi)$. Then
\begin{align*}
    (T^*)^n T^n=\mathbb{E}\circ \mathbb{E}_n,
\end{align*}
where $\mathbb{E}_n$ and $\mathbb{E}$ are normal faithful conditional expectations on $M$ and $(\mathbb{E}_n)$ is decreasing.
\end{thm}

\begin{rem}
For  a given a factorable map $T$ on  a noncommutative tracial measure space $(M,\tau)$, consider the factorozation as in Theorem \ref{Thm:FactoMap}
\begin{align*}
	(T^*)^n T^n=\mathbb{E}\circ \mathbb{E}_n,
\end{align*}
where $\mathbb{E}_n$ and $\mathbb{E}$ are normal faithful conditional expectations on $M$ and $(\mathbb{E}_n)$ is decreasing. Now  for $n \in \N$,  consider $ A_n = \mathbb{E}\circ \mathbb{E}_n $. Since both $\E$ and $ \E_n $ are $\tau$-preserving and positive contraction on $M$, so, by Proposition \ref{ext}, it extends to $L^\Phi(M, \tau)$, where, $\Phi$ is Orlicz function.  In particular $A_n$ extends to all $L^p(M, \tau)$ for $p \geq 1$.
\end{rem}

\begin{lem}\label{Lem:ContinuousA}
	The map $A_n=\E\circ\mathbb{E}_n : L^\Phi(M, \tau)\rightarrow L^0(M, \tau)$ is positive, additive and continuous.  
\end{lem}
\begin{proof}
First note that  $\E$ and $\mathbb{E}_n$  are $\tau$-preserving,  additive,  positive and contraction on $M$. Thus, it extends to $L^\Phi(M, \tau)$ (see  the Proposition \ref{ext}) and the map $ A_n : L^\Phi(M, \tau) \ri L^\Phi(M, \tau) $ is also additive,  positive and contraction. So, the map $A_n :L^\Phi(M, \tau) \rightarrow L^0(M, \tau)$ is also positive and additive. Further, as 
$L^\Phi(M, \tau)$ is continuously injects into $L^0(M, \tau)$ $($see \cite[Proposition 5.44]{GL20}$)$, so, it follows that the map $A_n$ from $L^\Phi(M, \tau)$ to $L^0(M, \tau)$ is continuous, where $L^\Phi(M, \tau)$ is equipped with the topology induced by the norm $\norm{\cdot}_\Phi$ and $L^0(M, \tau)$ is equipped with measure topology.
\end{proof}

\begin{thm}\label{Thm:buem} Let $ p > 1$ and  $T$ be a factorable operator on  a noncommutative tracial measure space $(M, \tau)$. Suppose 
\begin{align*}
	(T^*)^n T^n=\mathbb{E}\circ \mathbb{E}_n,
\end{align*}
where $\mathbb{E}_n$ and $\mathbb{E}$ are faithful  normal conditional expectations on $M$ and $(\mathbb{E}_n)$ is decreasing. Now  for $n \in \N$,  the map  $ A_n: L^p(M, \tau) \ri L^0(M, \tau)$  is b.e.u.m on $L^p(M, \tau ) $.

\end{thm}

\begin{proof}
We note that the map 
$A_n :L^p(M, \tau) \rightarrow L^0(M, \tau)$ are positive, continuous and additive, where $L^\Phi(M, \tau)$ is equipped with the topology induced by the norm $\norm{\cdot}_\Phi$ and $L^0(M, \tau)$ is equipped with measure topology.

Using Junge's Doob maximal inequality in the proof of  \cite[Lemma 4.3]{Ad06}, it is established that the sequence of maps   $(A_n)$ satisfies the  strong type $(p,p)$ with a constant $C_p=c/((p-1)^2)$, i.e.,
\begin{align*}
	\norm{{\sup_n}^+ A_n x}_p\leq \frac{c}{(p-1)^2}\norm{x}_p, \quad \text{ for all }x\in L^p(M)_+,
\end{align*}
where $c$ is a universal constant.  Then it follows that 
$(A_n)$ is p.b.e.b on $L^p(M, \tau)$. Indeed, suppose $ x \in L^p(M, \tau)_+$, then there exists 
$ a \in L^p(M, \tau)_+$ such that 
$$ A_n(x) \leq a \text{ and } \norm{ a } \leq C_p  \norm{ x }.$$
As $ a \in L^p(M, \tau)_+$, so,  for given $\eps > 0$, there exits $e \in \CP(M)$ such that $ \tau(1-e) \leq \eps $ and $ \norm{ eae } < \infty$. Consequently, we have 
\begin{enumerate}
	\item $\tau(1-e) \leq \eps $ and 
	\item $\underset{n \in \N}{ \sup}\norm{ e A_n(x)e } \leq \norm{ eae} < \infty $.
\end{enumerate}
For arbitrary $x \in L^p(M, \tau)$, we write $ x = x_1 -x_2 + i( x_3-x_4)$ where $ x_j  \in L^p(M, \tau)_+$ for $ j = 1, 2, \cdots, 4$. As $ x_j \in L^p(M, \tau)_+$, so for given $ \eps >0$, there exists $ e_j \in \CP(M)$ such that 
\begin{enumerate}
	\item $ \tau(1-e_j) \leq \frac{\eps}{4}$ for $ j = 1, 2, \cdots, 4$ and 
	\item $\norm{ e_j A_n(x_j)e } < \infty $.
\end{enumerate}
 Now suppose $ e_x = \displaystyle \bigwedge_{j = 1}^4 e_j $. Then observe that  $ \tau(1-e_x ) \leq \eps $ and  
\begin{align*}
	\norm{ e_x A_n(x)p_x} = \norm{ e_x A_n((x_1-x_2) +i ( x_3-x_4))e_x}\\
	\leq \sum_{j = 1}^4 \norm{ e_x A_n(x_j)e_x}
	\leq \sum_{j = 1}^4 \norm{ e_j A_n(x_j)e_j}  < \infty.
\end{align*}
Thus this proves that the sequence of map $(A_n)$ is \emph{p.b.u.b}.  
 Therefore, by Theorem \ref{Thm:BoundedMeasure},  it follows that $(A_n)$ is  \emph{b.e.u.m} on $L^p(M, \tau)$.

\end{proof}

\begin{defn}
Let $ p >0$, then a  function  $ \Phi: [0, \infty ) \ri [0, \infty )$  is called $p$-convex if
the function $ \Phi^{  \frac{1}{p}  } :   [0, \infty ) \ri [0, \infty )$ defined by $  \Phi^\frac{1}{p} (t) =  ( \Phi (t))^{ \frac{1}{p}} $  for $ t \in [0, \infty ) $ is a convex function. 
\end{defn}
\begin{exmp}
	\begin{enumerate}\addtolength{\itemsep}{0.15cm}
		\item Let $ p > 1$ and   consider the function $ \Phi: [0, \infty) \ri  [0, \infty) $ defined by $ \Phi(t) = t^p $ for $ t \in [0, \infty)$. Then $\Phi$  is $p$-convex function.

		\item Suppose $ p = \frac{10}{9} >1$. Consider the function $ \Phi: [0, \infty) \ri  [0, \infty) $ defined by $ \Phi(t) = t\log(1+t) $ for $ t \in [0, \infty)$. 
		Then note that the function $ \Phi^{ \frac{1}{p}}(t)= t^{0.9} ( \log(1+t))^{ 0.9}		$ for $t \in [0, \infty )$ is a convex function. Therefore,  $\Phi$  is a $p$-convex function.

		\item Consider the  function $ \Phi(t) = t $ for $ t \in [0, \infty )$. Then we note that $\Phi$ is not a $p$-convex function for any $p >1$. 
	\end{enumerate}
\end{exmp}

The following is one of the important lemma for proving the Rota theorem.
\begin{lem}\label{Lem:BUEMA} Let $\Phi$ is a $p$-convex Orlicz function for some $p > 1$. 
Then we have the following. 
\begin{enumerate}
	\item The sequence  $(A_n)$ is  \textit{b.u.e.m.} at $0$ on $L^\Phi(M, \tau)$. 
	\item In addition suppose the function $ \widetilde{ \Phi }(t) = \Phi( \sqrt{ t})$ for $ t \in [0, \infty )$, is also $ p$-convex, then the sequence $(A_n)$ is \textit{a.u.e.m.} at $0$.
\end{enumerate}	

\end{lem}
%%%%%%%%%%%%%%%%%%%%%%%%%%%%%%%%%%%%%%%%%%%%%%%%%%%%%%%%%%%%%%%%%%%%%%%%%%%%%%%%%%%%%%%%%%%%%%%%%%%%%%%%%%%%%%%%%%%%%%%%%%%%%%%%%%%%%%%%%%%%%%%%%%%%%%%%%%%%%%%%%%%%%%%%%%%%%%%%%%%%%%%%%%%%%%%%%%%%%%%%%%%%%%%%%%%%%%%%%%%%%%%%
\begin{proof}
	As $ \Phi$ is p-convex Orlicz function, so the function $  \Phi^\frac{1}{p}(\lambda)  = { \Phi(  \lambda ) }^\frac{1}{p} $ for $ \lambda \in [ 0, \infty )$ is a Orlicz function. Thus, for given 
	$\delta >0$, there exists a $t > 0 $ such that 
	$$ t \cdot \Phi^\frac{1}{p}(\lambda) \geq \lambda, \text{ whenever } \lambda \geq \frac{\delta}{2}.$$
	i.e, we have $ t \cdot \Phi^\frac{1}{p}(\lambda)  \geq \lambda, \text{ whenever } \lambda \geq \frac{\delta}{2}.$
	Now suppose $ x \in M \cap L^1(M, \tau)_+$, then  we have the following 	
	\begin{align}\label{less}
	\nonumber 	x &=  \int_0^\infty  \lambda d e_\lambda  =  \int_0^{ \frac{\delta}{2} }  \lambda d e_\lambda  +  \int_{ \frac{\delta}{2} }^\infty  \lambda d e_\lambda \\
	\nonumber		&\leq \int_0^{ \frac{\delta}{2} }  \lambda d e_\lambda  +   t \cdot \int_{ \frac{\delta}{2} }^\infty   \Phi^\frac{1}{p}(\lambda) d e_\lambda \\
		&= x_\delta + t \cdot  \Phi^\frac{1}{p}(x), 
	\end{align}
where $ x_\delta =  \int_0^{ \frac{\delta}{2} }  \lambda d e_\lambda $.  	Suppose $ y =   \Phi^\frac{1}{p}(x)$, then we observe that 
	\begin{align*}
		\tau( y^p)&= \tau( \Phi( x) ) \leq \norm{ x}_\phi.
		%	&=\tau(   \Phi^2( \sqrt{x}))\\
		%	&\leq \norm{ \sqrt{x}}_{ {\Phi}^2} = \sqrt{ \norm{ x}_\Phi }
	\end{align*}	
	Thus, we have $ \norm{ y}_p  \leq  ( \tau( \Phi( x) ))^\frac{1}{p} \leq 	\norm{ x}_\phi^{\frac{1}{p}} $, i.e, $ \norm{  \Phi^\frac{1}{p}(x)}_p \leq \norm{ x}_\phi^{\frac{1}{p}} .$\\
	
	We know that the map $ A_n : L^p(M, \tau)_+ \rightarrow L^0(M, \tau)$ is \emph{b.u.e.m} at $0$ (see Theorem \ref{Thm:buem}), so, for given $\delta > 0$ and $ \eps >0$, there exists a $\gamma \in (0, 1) $ such that  for all  $ x \in L^p(M, \tau)$ with $ \norm{x}_p \leq \gamma $, there exists a projection $ e \in \CP(M) $ such that 
	\begin{enumerate}\addtolength{\itemsep}{0.15cm}
		\item $\tau(1-e) \leq \eps $ and 
		\item $ \underset{ { n \in \N}  }{\sup}  \norm{eA_n(x)e} \leq \frac{\delta}{2t}$
	\end{enumerate}
	Now consider  $ \gamma_1 = \gamma^p $ and  let $ x \in M \cap L^1(M, \tau)_+ $ such that $ \norm{ x}_\Phi \leq \gamma_1$. Then by first part,  observe that 
	\begin{enumerate}\addtolength{\itemsep}{0.15cm}
		\item  $ \norm{ A_n(x_\delta ) } \leq \frac{\delta}{2} $ for all $ n \in \N$ and 
		\item $ \norm{ \Phi^\frac{1}{p}(  x )}_p \leq	\norm{ x}_\Phi^{\frac{1}{p}} \leq \gamma_1^{  \frac{1}{p}}= \gamma$.  
	\end{enumerate}
	Hence, there exists a projection $ e \in \CP(M) $ such that 
	$$  \underset{ { n \in \N}  }{\sup}  \norm{ e A_n( \Phi^\frac{1}{p}(  x )) e} \leq \frac{\delta}{2t}.$$
	Subsequently, by Eq. \ref{less}, it follows that 
	\begin{align*}
		\underset{ { n \in \N}  }{\sup}  \norm{ e A_n( x)e }  &\leq \underset{ { n \in \N}  }{\sup}  \norm{ e A_n( x_\delta )e }   + t \cdot \underset{ { n \in \N}  }{\sup}  \norm{ e A_n(\Phi^\frac{1}{p}(  x ) )e }  \\
		& \leq \frac{\delta}{2 } + t \cdot \frac{\delta}{ 2 t} = \delta.
	\end{align*}
Thus, $(A_n)$ is \emph{b.u.e.m} at $0$ in $M \cap L^1(M, \tau)$. Consequently, with the help of Theorem \ref{Thm:BUEMonClosure} and Theorem \ref{Thm:BoundedMeasure} , we conclude that $(A_n)$ is \emph{b.u.e.m} at $0$ in $L^\Phi(M, \tau)$.

	For the proof of $(2)$, fix $ \eps > 0 $ and $ \delta > 0$. Since $ \widetilde{ \Phi }$ is $p$-convex so by part $(1)$,  we note that the sequence is $(A_n)$ is  \textit{b.u.e.m.} at $0$ on $L^{\widetilde{ \Phi } }(M, \tau)$. Hence, there exists a $ \gamma > 0$ such that for given $ x \in L^{\widetilde{ \Phi } }(M, \tau)$ with $ \norm{ x }_{ \widetilde{ \Phi }   } \leq \gamma $, there exists a projection $ e\in \CP(M)$ such that 
\begin{enumerate} \addtolength{\itemsep}{0.15cm}
		\item $\tau(1-e) \leq \eps $ and 
		\item $ \underset{ { n \in \N}  }{\sup}  \norm{eA_n(x)e} \leq \delta^2.$ 
	\end{enumerate}
Now  suppose $ x \in L^\Phi(M, \tau)_+ $ with $ \norm{ x } \leq \sqrt{\gamma } $,   then we note that 
$ x^2 \in  L^{\widetilde{ \Phi } }(M, \tau)$ and observe that 

\begin{align*}
\norm{x^2 }_{  \widetilde{ \Phi } } &=\text{inf}\{ \lambda > 0 : \widetilde{ \Phi }\left( \frac{x^2}{\lambda} \right)    \leq 1    \}\\
&=\text{inf}\{ \lambda > 0 : \Phi\left( \frac{x}{\sqrt{\lambda}}\right)    \leq 1    \}\\
&=\text{inf}\{ \lambda^2 > 0 : \Phi\left( \frac{x}{\lambda}\right)    \leq 1    \}\\
&=\norm{ x }^2_\Phi  \leq \gamma.
\end{align*}
Hence, there exists a projection such that 
\begin{enumerate} \addtolength{\itemsep}{0.15cm}
	\item $\tau(1-e) \leq \eps $ and 
	\item $ \underset{ { n \in \N}  }{\sup}  \norm{eA_n(x^2)e} \leq \delta^2.$ 
\end{enumerate}
Then, we use the Kadison's inequality and obtain the following. 
\begin{align*}
\left(   \underset{ { n \in \N}  }{\sup}  \norm{A_n(x)e}  \right)^2 &=  \underset{ { n \in \N}  }{\sup}  \norm{A_n(x)e}^2 \\
&= \underset{ { n \in \N}  }{\sup}  \norm{e A_n(x) A_n(x)e}  \\
&\leq  \underset{ { n \in \N}  }{\sup}  \norm{e A_n(x^2)e} \leq \delta^2. 
\end{align*}
Thus, we obtain that the sequence $(A_n)$ is \textit{u.e.m.} at $0$ on $L^\Phi(M, \tau)_+$ and consequently, it is \textit{u.e.m.} at $0$ on $L^\Phi(M, \tau)$.

\end{proof}

\begin{thm}[\textbf{Noncommutative Rota's theorem}]\label{Thm:NCRotaOrlicz}
	Let $(M, \tau)$ be a noncommutative tracial measure space  and  $\Phi$ be an Orlicz function.  Suppose   $T:(M,\tau)\rightarrow (M,\tau)$ is  a factorizable $\tau$-preserving Markov operator and   $\Phi$ is  $p$-convex for some $p >1$, then  for any $x\in L^\Phi(M, \tau)$, the sequence  $\big(T^n (T^*)^n( x) \big)$ converges in \emph{b.a.u.}\\
	Furthermore, if the function $ \widetilde{ \Phi }(t) = \Phi( \sqrt{ t})$ for $ t \in [0, \infty )$, is  also  $ p$-convex, then for all $x\in L^\Phi(M, \tau)$, the sequence  $\big(T^n (T^*)^n( x) \big)$ converges in \emph{a.u.}
\end{thm}

\begin{proof}
	Firstly, note that by Theorem \ref{Thm:FactoMap}, one has
	\begin{align*}
		(T^*)^n T^n (x)=\mathbb{E}\circ \mathbb{E}_{n} (x), \quad x\in M,
	\end{align*}
	where $\mathbb{E}_n$ and $\mathbb{E}$ are normal faithful conditional expectations on $M$ and $(\mathbb{E}_n)$ is decreasing.	Note that the above formula can be extended to $L^\Phi(M, \tau)$ by using the Proposition \ref{ext}. Write  $A_n:=(T^*)^n T^n=\mathbb{E}\circ \mathbb{E}_{n}$. Then by Lemma \ref{Lem:BUEMA}, it follows that the sequence of maps $(A_n)$ is  \emph{b.u.e.m}  on $L^\Phi(M, \tau)$. Therefore, by Theorem \ref{Thm:BanachPrincipleAu}, it follows that for any $x\in L^\Phi(M)$, the sequence $(A_n(x))$   converges in \textit{b.a.u.}. Equivalently, it proves that for any $x\in L^\Phi(M)$,  the sequence $( T^n (T^*)^n (x))$ converges in \textit{b.a.u.}

By using  similar argument as given in part $(2)$ of  Lemma \ref{Lem:BUEMA},   we obtain that the sequence $(A_n(x)) $  converges in \textit{a.u} for all $ x \in L^\Phi(M, \tau)$. This completes the proof.
	
\end{proof}

\section{Non commutative  Ergodic theorems for  free group and semigroup actions}\label{Sec:MainResults}
In this section, our aim is to study noncommutative maximal inequalities and ergodic theorems associated to action by free group or semigroups on noncommutative spaces. Actually, we study it  for more  general context.  As a particular case the  noncommutative maximal inequalities and ergodic theorems associated to action by free group or semigroups on noncommutative spaces can be obtained. First we set the preliminary background.
We begin with the following notations. We write 
\begin{enumerate}
	\item $[m]: = \{ 1, 2, \cdots, m \} $
	
	\item $ [-m..m]=  \{-m, -(m-1), \cdots, -1, 1, 2, \cdots, m \} $
	
\end{enumerate}
Suppose $I \in  \{ [m], [-m..m] \}$, i.e,  $I$ is either $[m]$ or $[-m..m]$. Then we consider the following: 
\begin{enumerate}
\item $I^\N : = \{ \omega= ( \omega_1, \omega_2, \cdots ): ~ \omega_j \in I \text{ for all } j \in \N \}$ set of all one sided sequence in the  symbols  of elements form $I$ and 
	\item $I^F : = \{ ( \omega_1, \omega_2, \cdots,\omega_n ): ~ n \in \N, \omega_j \in I \text{ and  for all } 1 \leq j\leq n  \}$ the set of all finite tuple of   symbols  of elements form $I$.
	For simplicity of writing  we also identify 
	$I^F : = \{  \omega_1 \omega_2 \cdots \omega_n : ~ n \in \N, \omega_j \in I \text{ and  for all } 1 \leq j\leq n  \}$, i.e, it is the collection of finite words in $I$.
\end{enumerate}
We consider the shift map $ \sigma: I^\N \ri I^\N$ which is defined as 
$$ \sigma(\omega)(n ) = \omega_{ n+1}, \text{ for } \omega \in I^\N.$$\\
Suppose $P=(p_{ij})_{i,j\in I}$  is a stochastic matrix whose rows and columns are indexed by the  finite set $I$. Suppose $( p_i)_{ i \in I}$ is a  stationary distribution with transition matrix $P$, i.e, it is a probability measure on $I$ such that 
$$ p_j=\underset{ i \in I}{\sum} p_i p_{ij}.$$  
Assume that  $p_i>0$ for all $i\in I$. 
Let $\mu$ be a Borel $\sigma$-invariant Markov  measure on $I^\N$ with initial distribution $(p_1,\ldots,p_m)$ and transition 
probability matrix $P=(p_{ij})_{i,j\in I}$. 
For $\omega =( \omega_1, \omega, \cdots, \omega_n) \in  I^F$, we denote 
$$C(\omega) = \{ (\omega_1, \omega, \cdots, \omega_n, v_1, v_2, \cdots ):   ( v_1, v_2, \cdots )\in I^\N \},
$$
i.e, $C(\omega)$ is the set of  sequences  words in $I$ starting with $\omega$. We further write 
$$\mu(\omega):=\mu(C(\omega)).$$
Suppose $ \omega = i_1 i_2 \cdots i_n \in I^F $, then we note that 
\begin{align}\label{stocha}
\mu( \omega) = \mu(i_1i_2 \cdots i_n) = \mu(i_1)p_{i_1 i_2}p_{i_2 i_3}\cdots p_{i_{n-1} i_n}.
\end{align}
Now suppose $ X$ is Banach space and for each  $ i \in I$,  let $  \alpha_i  : X \ri X$ be a linear operator. Then for $\omega =( \omega_1, \omega, \cdots, \omega_n) \in  I^F$, we write 
$$ \alpha_{\omega }= \alpha_{ \omega_n} \alpha_{ \omega_{n-1}}   \cdots \alpha_{ \omega_1}.$$
Let  $ \abs{\cdot}: I^F \ri I^F$ be a function, defined by 
$$\abs{\omega}= n \text{ for }  \omega =( \omega_1, \omega, \cdots, \omega_n) \in  I^F$$
i.e, it counts the number of tuple.
Further, for simplicity of notation for  $j \in I$,  and $ \omega=(\omega_1, \omega_2, \cdots, \omega_n) \in I^F$, we write $\omega j$  for the element $  (\omega_1, \omega_2, \cdots, \omega_n, j )$, i.e, 
$$ \omega j = (\omega_1, \omega_2, \cdots, \omega_n, j ).$$
Now consider the following spherical average
\begin{align*}
	S_n=\sum_{\abs{\omega}=n}\mu(\omega)\alpha_\omega,
\end{align*}
and the averages of $S_k$ over $k=0,1\ldots,n-1$, by
\begin{align}\label{Eq:AverageOfAverage}
	A_n=\frac{1}{n}\sum_{k=0}^{n-1}S_k.
\end{align}

For each $ i \in I$, let us further consider the following operators:
\begin{align}\label{Eq:AverageOfIthAverage}
	\nonumber S_n^{(i)}&=\sum_{\{\omega:\abs{\omega}=n,~ \omega_n=i\}}\mu(\omega)\alpha_\omega \text{ and } \\
	A_n^{(i)}&=\frac{1}{n}\sum_{k=0}^{n-1}S_k^{(i)},
\end{align}
with  the convention that  $\alpha_\omega=1$, the identity operator on $X$, whenever $\abs{\omega}=0$. Now readily  note that  
\begin{align}\label{Eq:SumOfCni}
	A_n=\sum_{i\in I} A_n^{(i)}.
\end{align}
Our aim is to study the convergence of $S_n$ and $A_n$ in noncommutative  $L^p$-spaces and Orlicz spaces. \\
The following lemma is important to us. 
\begin{lem}\label{Lem:SnPlusOneEquation}
	For each $j\in I$ and for any positive integer $n$, we have the following equality:
\begin{enumerate}
	\item $S_{n+1}^{(j)}=\displaystyle\sum_{i\in I}p_{ij}\alpha_j S_n^{(i)} \text{ and }$\\
		
\item  $\frac{S_{n+1}^{(j)}}{p_j}=\displaystyle\sum_{i\in I} \frac{p_i p_{ij}}{p_j}\alpha_j \left(\frac{S_{n}^{(j)}}{p_i}\right)$.	
\end{enumerate}
\end{lem}
\begin{proof}
For the assertion $(1)$, 	note that 
	\begin{align*}
		S_{n+1}^{(j)}&=\sum_{\{\omega:~\abs{\omega}=n+1, \omega_{n+1}=j\}}\mu(\omega)\alpha_\omega\\
		&=\sum_{\{\omega:~\abs{\omega}=n\}}\mu(\omega j)\alpha_j\alpha_\omega\\
		&=\alpha_j\sum_{i\in I} \sum_{\{\omega:~\abs{\omega}=n, ~\omega_n=i\}}\mu(\omega j)\alpha_\omega\\
		&=\alpha_j\sum_{i\in I}\sum_{\{\omega:~\abs{\omega }=n-1\} }\mu(\omega i j)\alpha_{\omega i}\\
		&=\sum_{i\in I} p_{ij}\alpha_j S_{n}^{(i)}, ~(\text{as } \frac{\mu(\omega i j )}{ \mu(\omega i)} = p_{ij}, \text{   see Eq. } \ref{stocha}).
	\end{align*}
This proves $(1)$.  The second equality follows easily from the first equation. This completes the proof.
\end{proof}

Following \cite{Bu00}, we consider a new operator 
\begin{align}\label{Eq:DirectSumOperator}
T: \underset{i\in I}{\oplus} X &\rightarrow \underset{i\in I}{\oplus} X \text{ by }\\
\nonumber T(x_1,\ldots, x_m)&=(y_1\ldots,y_m), \quad\text{ where } y_j=\sum_{i\in I}\frac{p_i p_{ij}}{p_j} \alpha_j (x_i).
\end{align}
Then we have the following relations.
\begin{lem}\label{Lem:DiagonalOperatorTn}
	Given any $x\in X$ and $n\in\mathbb{N}$, we have
	\begin{enumerate}
		\item
		$T^n(x,\ldots, x)=\left(\frac{S_n^{(j)}(x)}{p_j}\right)_{ j \in I}
	$
	and
\item 	$
		\displaystyle\frac{1}{n}\sum_{k=0}^{n-1}T^n(x,\ldots,x)=\left(\frac{A_n^{(j)}(x)}{p_j}\right)_{ j \in I}.$
\end{enumerate}
\end{lem}
\begin{proof}
We use induction to prove our assertion. For $(1)$, we note that 
\begin{align*}
T(x, \cdots, x) &= \left( \sum_{i\in I}\frac{p_i p_{ij}}{p_j} \alpha_j (x)\right)_{ j \in I}\\
&= \left(\frac{S_1^{(j)}(x)}{p_j}    \right)_{ j \in I}. 
\end{align*}
So, the result is true for $n =1$. Now 
suppose $T^k( x, \cdots, x) = 	\left(\frac{S_k^{(j)}(x)}{p_j} \right)_{ j \in I}$. Then we have 
\begin{align*}
T^{ k+1}( x, \cdots, x) &= T  	\left(\frac{S_k^{(j)}(x)}{p_j}\right)_{ j \in I}.\\
\end{align*}
So, by definition of $T$ (see Eq. \ref{Eq:DirectSumOperator}), the $j$-th coordinate of $T  	\left(\frac{S_k^{(j)}(x)}{p_j}\right)_{ j \in I}$ will be 
$$ \sum_{ i \in  I} \frac{p_i p_{i, j}}{ p_j} \alpha_j \left(  \frac{S_k^{(i)}(x)}{p_i}  \right).$$
Then, we note that 
\begin{align*}
 \sum_{ i \in  I} \frac{p_i p_{i, j}}{ p_j} \alpha_j \left(  \frac{S_k^{(i)}(x)}{p_i}  \right)    
&= \sum_{ i \in  I}  p_{i, j} ~\alpha_j \left(  \frac{S_k^{(i)}(x)}{p_j}  \right)    \\
&=\frac{S_{k+1}^{(j)}(x)}{p_j}
\end{align*}	
This completes the proof of $(1)$ and the proof $(2) $ is straightforward. 
\end{proof}

The following lemma is a key to our main result.

\begin{lem}\label{Lem:TisContraction}
	For $i\in I$, let $\alpha_i\in \mathbf{B}(L^p(M,\tau))$ be a positive contraction.
	Consider the linear operator $T$ defined on $\underset{ i \in I}{\oplus} L^p(M,\tau)$ as in Eq. \eqref{Eq:DirectSumOperator}, where $\underset{ i \in I}{\oplus}  L^p(M,\tau)$ equipped with the following norm
	\begin{align*}
		\norm{(x_j)}_p=\left(\sum_{i\in I}p_i\norm{x_i}_p^p \right)^\frac{1}{p},\quad (x_j) \underset{ i \in I}{\oplus}  L^p(M,\tau).
	\end{align*}
	Then $T$ is a positive contraction.
\end{lem}

\begin{proof}
	Let $(x_j) \in \underset{ j \in I }{ \oplus} L^p(M,\tau)$ and suppose 	$(y_j)=T(x_j)$.
Then note that
	\begin{align*}
		\norm{y_j}_p^p
		=&\norm{\sum_{i\in I} \frac{p_i p_{ij}}{p_j}\alpha_j(x_i)}_p^p\\
		\leq& \left(\sum_{i\in I} \frac{p_i p_{ij}}{p_j}\norm{\alpha_j(x_i)}_p\right)^p\\
		=& \left(\sum_{i\in I} \left(\frac{p_i p_{ij}}{p_j}\right)^\frac{1}{q} \left(\frac{p_i p_{ij}}{p_j}\right)^\frac{1}{p}\norm{\alpha_j(x_i)}_p\right)^p\\
		\leq& \left(\sum_{i\in I} \frac{p_i p_{ij}}{p_j}\right)^\frac{p}{q} \left(\sum_{i\in I}\frac{p_i p_{ij}}{p_j}\norm{\alpha_j(x_i)}_p^p\right), (\text{by Holder})\\
		=&\sum_{i\in I}\frac{p_i p_{ij}}{p_j}\norm{\alpha_j(x_i)}_p^p\\
		\leq & \sum_{i\in I}\frac{p_i p_{ij}}{p_j}\norm{x_i}_p^p.
	\end{align*}
	Consequently,
	\begin{align*}
		\norm{(y_j)}_p^p	
		=\sum_{j\in I} p_j\norm{y_j}_p^p
		\leq &\sum_{j\in I} \sum_{i\in I} p_i p_{ij}\norm{x_i}_p^p\\
		=& \sum_{i\in I} \left( p_i \norm{x_i}_p^p\right) \left(\sum_{j\in I}p_{ij}\right)\\
		=&\sum_{i\in I}  p_i \norm{x_i}_p^p\\
		=&\norm{(x_j)}_p^p.
	\end{align*}
	Thus, $T$ is a contraction. The positivity of $T$ follows trivially. We omit the details.
\end{proof}

\subsection{Ergodic theorems for spherical averages for free group actions}

In this subsection we discuss  the \emph{b.a.u} and \emph{a.u} convergence of the spherical averages of actions of a free group $\F_m$ on noncommutative spaces.

Let us now recall the following setting. 
Let $\F_m$ be the  free group on $m$ generators $ \{ a_1, a_2, \cdots, a_m \}$. In this subsection of we consider 
$$ I = \{ -m, \cdots, -1, 1, \cdots, m \}.$$
Let $(M, \tau)$ be a tracial measure space. For $ 1 \leq i \leq m $,  suppose $ \alpha_i : M \ri M $ is a normal  automorphism  corresponding  to generator $a_i$ such that $ \tau \circ \alpha_i = \tau $ and we set $ \alpha_{-i} = \alpha_i^{-1} $. \\
We take $P = (  p_{ ij})_{ i, j \in I}$ where $ p_{i, j } = 0 $ for $ i = -j $ and otherwise  $ p_{ i, j }= \frac{1}{2m-1}$ for $i, j \in I$.  Suppose $(p_i)_{ i \in I}$ is the stationary distribution with $p_i = \frac{1}{2m}$ for all $i \in I$.\\
Now consider  $R = \underset{ i \in I}{\oplus} M$ and recall the operator $ T : R \ri R$ from  Eq. \ref{Eq:DirectSumOperator}.
Let $x \in R$ and $x = (x_{-m},\ldots, x_{-1}, x_1, \ldots,  x_m) $.
 Suppose  $$T(x_{-m},\ldots, x_{-1}, x_1, \ldots,  x_m) =(y_{-m},\ldots, y_{-1}, y_1, \ldots,  y_m),$$ 
we write $y_j =T(x)_j$ and it is given by 
$$ T(x)_j=\frac{1}{2m-1} \displaystyle\sum_{i\in I,~~ i \neq -j } \alpha_j (x_i).$$
Suppose  $ \omega = i_1 i_2 \cdots i_n \in I^F $ such that $ i_k + i_{ k+1} \neq 0 $ for $k = 1, 2, \cdots, n-1$, then we note that 
\begin{align*}
\mu( \omega) &= \mu(i_1i_2 \cdots i_n) = \mu(i_1)p_{i_1 i_2}p_{i_2 i_3}\cdots p_{i_{n-1} i_n}\\
&=  \frac{1}{2m  (2m-1)^{n-1}} 
\end{align*}
We also recall that 
\begin{align*} 
S_n^{(i)} &=\sum_{\{\omega:\abs{\omega}=n,~ \omega_n=i\}}\mu(\omega)\alpha_\omega.
\end{align*}
Let $ \mathbb{S}_n = \{ i_1 i_2 \cdots i_n \in I^F :  \text{ such that } i_k + i_{ k+1} \neq 0, \text{ for } k = 1, 2, \cdots, n-1 \}$   and suppose $ \abs{ \mathbb{S}_n }  = \text{ denotes the  cardinality of the  set } \mathbb{S}_n$.  In fact in this case $$\abs{ \mathbb{S}_n} =\frac{1}{2m ( 2m-1)^{n-1}}.$$ 

Thus, for $ x \in M$,  we note that 
 \begin{align*}
 \frac{1}{\abs{ \mathbb{S}_n}} \sum_{ \omega \in  \mathbb{S}_n } \alpha_\omega(x) &= \frac{1}{2m ( 2m-1)^{n-1}}\sum_{ \omega \in \mathbb{S}_n  } \alpha_\omega(x)\\
&= \sum_{ i \in I} \sum_{\{ \omega \in I^F: \abs{\omega} = n, ~ \omega_n = i \} } \mu( \omega) \alpha_\omega(x) \\
&= \sum_{ i \in I } S_n^{(i)} = S_n(x)\\
\end{align*}

Further, we note that 
\begin{align*}
S_n(x) &=  \sum_{ i \in I } S_n^{(i)}(x)=   \frac{1}{2m} \sum_{ i \in I }  \frac{S_n^{(i)}(x)}{ p_i}\\
&= \frac{1}{2m} \sum_{ i \in I }  T^n(\tilde{x})_i.
\end{align*}
Thus, we remark that to study   the convergence of $ (S_n(x))$ we need to  study the  convergence  of the sequence $(T^n(x))$. \\
Let $ y= ( y_i) \in R$, then the adjoint of $T$ is given by 
$$ 
T^*(y)_j = \frac{1}{p_j} \sum_{ i \in I} p_i ~~p_{ ij} \alpha_{-i}( y_j).
$$
Now consider a unitary $U : L^2(R) \ri L^2(R) $ which is extension of the following trace preserving map
$$ R \ni (y_j) \longrightarrow \left( \alpha_j( y_{ -j})    \right). $$
Then we note that $U^2 = 1$ and $U TU = T^*$.
Then we have  the following relations (see \cite{Ad06}): 
 
	\begin{align}\label{Relation-}
	\nonumber&(1)~~(T^*)^{n}  T^n= \frac{2m -2}{2m-1} UT^{ 2n -1} + \frac{1}{2m -1}(T^*)^{n-1} T^{n-1} \\ 
	 &(2)~~T^{2n -1} = \frac{2m -1}{2m-2} U   (T^*)^{n}  T^n  +  \frac{1}{2m -2}  U (T^*)^{n-1} T^{n-1}. 
\end{align}

Suppose $\F_m^{ (2)}$ be the subgroup of $\F_m$  consisting of even words. Let $ M^{ (2)} = \{ x \in M : \alpha_{ \omega} (x) = x, \text{ for all } \omega \in \F_m^{ (2)} \}$ and $\E^{(2)}: M \ri A$ be the faithful normal $\tau$-preserving conditional expectation. So, it will extend to $ L^\Phi(M, \tau)$ (see Proposition \ref{ext})  and we denotes the extension by same notation $\E^{(2)}$.  
 Consider the Orlicz function $ \Psi(t) = t \log(1+t) $, for $ t \in [0, \infty)$, then we write 
 $$ L\log L(M, \tau) = L^\Psi( M, \tau).$$

\begin{thm}[\textbf{Non-commutative Bufetov ergodic theorem}]\label{Thm-Ergo-Orlicz}
	Let $(M, \tau)$  be a noncommutative tracial measure space  and  $\Phi$ be an $p$-convex  Orlicz function for some $p >1$.
For each  $1 \leq  i \leq m  $, suppose    $\alpha_i:(M,\tau)\rightarrow (M,\tau),~~~$ is  a factorizable $\tau$-preserving Markov operators.   Then we have the following. 

\begin{enumerate}
	\item 
For any  $x\in L^\Phi(M, \tau)$, the sequence  $\big(S_{2n}(x) \big)$ converges in  \emph{b.a.u} to $ \E^{ (2)}(x)$.
\item In addition suppose the function $\widetilde{ \Phi } (t) = \Phi( \sqrt{t}) $ for $ t \in [1, \infty) $ is also $p$-convex. Then for  any  $x\in L^\Phi(M, \tau)$, the sequence  $\big(S_{2n}(x) \big)$ converges  in  \emph{a.u} to $ \E^{ (2)}(x)$.
\end{enumerate} 
In particular, $(1)$ holds for all $ x \in L \log L(M, \tau) $, i.e, the sequence $\big(S_{2n}(x) \big)$ converges in  \emph{b.a.u} for all $ x \in L \log L(M, \tau)$.
\end{thm}

\begin{proof}
	
	The convergence part of the proof  follows from the Theorem \ref{Thm:NCRotaOrlicz} and the Eq.\ref{Relation-}. 
	Only we need to  show the invariance of the limit. 
	
Indeed, for  $ x \in L^\Phi(M, \tau)$,  we note that the sequence $(S_{2n}(x) )$ converges in \emph{b.a.u}. Moreover, if $\Phi$ satisfies the condition $(2)$ of Theorem \ref{Thm-Ergo-Orlicz}, then $(S_{2n}(x) )$ converges in \emph{a.u}.  
In particular, if $ x \in M \cap L^1(M, \tau)$, then the sequence $(S_{2n}(x) )$ converges in \emph{a.u}.  
For  $ x \in M \cap L^1(M, \tau) $, note the sequence $ (  S_{2n}(x))$ is bounded (in $\norm{ \cdot}_\infty$) and  
consequently, for $ x \in M \cap L^1(M, \tau)$, the sequence $(S_{2n}(x) )$ converges in \emph{SOT}. Since $ \alpha_i$  for $ 1 \leq i \leq m$ are $ \tau$-preserving, so, further we have the following 
$$ \tau(y S_{2n} (x)) = \tau(yx) \text{ for all } y \in M^{(2)} \text{ and } x \in M \cap L^1(M, \tau).$$

 For $ x \in M \cap L^1(M, \tau)$, suppose $ \widehat{x}= \text{a.u-}\lim S_{ 2n}(x)$. Then for all $ y \in M^{ (2)}$, we note the following 
 \begin{align*}
\tau( y \E^{(2)}(x )) & =  \tau(y x), \text{ as } \tau\circ \E^{(2)} = \tau\\
&=  \lim  \tau(y S_{2n} (x)) = \tau(y \widehat{x}).
 \end{align*} 
Therefore, to deduce $ \widehat{ x } = \E^{ (2)} $, it is enough to show that  $ \alpha_{ \omega}( \widehat{x} ) = \widehat{x}$, for all $ \omega \in \F^{ (2)}$.
Since $ \text{SOT-}\lim S_{ 2n}(x)= \widehat{x} $, so, by using the following  following standard identity: 
$$ S_1^2 \circ S_{2n}  = \left( \frac{2m -1}{2m} \right)^2 S_{ 2n +2} +  \frac{2m -1}{2m^2} S_{ 2n} +   \frac{1}{4m^2}  S_{ 2n -2},$$
we obtain that $  S_1^2 ( \widehat{x}) = \widehat{x} $. Further, since $ \widehat{ x} \in L^2(M, \tau)$, so, using the strictly convexity of $ L^2(M, \tau)$,  we  obtain  that $$ \alpha_{ \omega}( \widehat{ x}) = \widehat{ x},  \text{ for all } \omega \in \F_m^{(2)}.$$
Therefore, $\text{ for all } x \in M\cap L^1(M, \tau)$, we obtain $$ \text{b.a.u-}\lim S_{ 2n}(x)= \text{a.u-}\lim S_{ 2n}(x) =\text{SOT-}\lim S_{ 2n}(x)= \widehat{x} = \E^{(2)}(x ).$$\\
Now if  $ \widehat{ x } =  \text{b.a.u-}\lim S_{ 2n}(x)$ for $ x \in L^\Phi( M, \tau) $, we like to  show that $\widehat{x} = \E^{(2)}(x ).$
Indeed, let $ x \in L^\Phi( M, \tau)_+ $ and since $ M \cap L^1(M, \tau)$ is dense in $ L^\Phi( M, \tau) $, so, there exists an increasing sequence $ (x_k) $ in  $ M \cap L^1(M, \tau)_+$ with $ x_k \leq x $ for all $ k \in \N$ (apply  Lemma \ref{dense} and possibly find a further  subsequence), such that 

\begin{enumerate}
	\item $  \underset{ k \ri \infty }{\lim} \norm{ x - x_k }_\Phi = 0 $, 
		\item  $ x = \text{b.a.u-}\underset{ n \ri \infty }{\lim}  x_k $ and  
		
		\item $S_{2n}(x)= \text{b.a.u-}\underset{ k \ri \infty }{\lim} S_{2n} (x_k)$ for all $n \in \N$.
\end{enumerate} 
 Suppose $ \widehat{ x} = \text{ b.a.u-} S_{2n}(x) $ and we like to show  that $ \E^{ (2)}(x) = \widehat{x}$. 
First we like to  note that $ (\widehat{x_k})$ converges to $ \widehat{x} $ in \emph{ b.a.u}.
Indeed, as $ \widehat{ x} = \text{ b.a.u-} S_{2n}(x) $, so, there exists a projection 
$e_1$ such that 
\begin{enumerate}
	\item $\tau( 1 - e_1) \leq \frac{ \eps }{ 4}$ and 
	\item $ \underset{ k \in \N }{\sup}  \norm{ e_1 (S_{2n}( x ) -\widehat{x}) e_1 } \leq \frac{\delta}{ 4}$ for all $n \in \N$.
\end{enumerate}
Again as $S_{2n}(x)= \text{b.a.u-}\underset{ k \ri \infty }{\lim} S_{2n} (x_k)$ for all $n \in \N$. So,there exists a projection $e_2$ such that 
\begin{enumerate}
\item $\tau( 1 - e_2) \leq \frac{ \eps }{ 4}$ and 
\item $ \underset{ k \in \N }{\sup}  \norm{ e_2 (S_{2n}( x -x_k) e_2 } \leq \frac{\delta}{ 4}$ for all $n \in \N$.
\end{enumerate}
Further, as $ \text{b.a.u-}\lim_{ n \ri \infty } S_{2n}(x_k) = \widehat{ x_k}.$ So, there exists a  sequence of projections $(e_k) $ such that 

\begin{enumerate}
	\item $\tau( 1 - p_k) \leq \frac{ \eps }{ 4^k}$ and 
	\item $ \underset{ n \in \N }{\sup}  \norm{ p_k (S_{2n}(x_k) - \widehat{x_k}) p_k } \leq \frac{\delta}{ 4}$ for all $n \in \N$.
\end{enumerate}

Suppose $ e =     e_1 \wedge e_2  \overset{\infty}{\underset{{k = 1} }{\wedge}}  p_k $ and note that 
$$ \tau(1-e) \leq  \tau( 1 - e_1) + \tau( 1 - e_2) + \sum_{ k =1}^\infty \tau( 1-p_k) \leq \frac{\eps}{2} +\sum_{ k =1}^\infty \frac{\eps}{2^{k+2}}\leq \eps.$$
Then for all $k \in \N$, we note that 
\begin{align*}
\norm{ e \left( \widehat{x_k} -\widehat{x} \right)e } &= \norm{ e \left( \widehat{x_k} - S_{2n}( x_k) + S_{2n}( x_k) - S_{2n}(x) +  S_{2n}(x) -\widehat{x} \right)e }\\
&\leq  \norm{ p_k \left( \widehat{x_k} - S_{2n}( x_k) \right)p_k } + 
\norm{ e_2 \left(  S_{2n}( x_k-x) \right)e_2 }
+ \norm{ e_1 \left(  S_{2n}(x) -\widehat{x} \right)e_1 }\\
& \leq \delta.
\end{align*}
Thus, $ \widehat{x} = \text{ b.a.u-} \lim  \widehat{x_k}$. But, we note that $ (x_k) \subseteq M \cap L^1(M, \tau)_+ $, so, by the first part we have $ \E^{(2)}( \widehat{x_k}  ) = \widehat{x_k}$ for all $ k \in \N$. Therefore, similar arguments can be used to obtain that $ \E^{(2)}( \widehat{x}  ) = \widehat{x}$.  This completes the proof.

 \end{proof}

%%%%%%%%%%%%%%%%%%%%%%%%%%%%%%%%%%%%%%%%%%%%%%%%%%%%%%%%%%%%%%%%%%%%%%%%%%%%%%%%%%%%%%%%%%%%%%%%%%%%%%%%%%%%%%%%%%%%%%%%%%%%%%%%%%%%%%%%%%%%%%%%%%%%%%%%%%%%%%%%%%%%%%%%%%%%%%%%%%%%%%%%%%%%%%%%%%%%%%%%%%%%%%%%%%%%%%%%%%%%%%%%%%%%

\subsection{ Ergodic theorems for averages of spherical averages for free group actions}

In this subsection we consider a noncommutative tracial  measure space $(M, \tau)$, where $\tau$ is a f.n semifinite trace $\tau$ on $M$ and let $1 \leq p<\infty$. For $n\geq 0$, let $T_n: L^p(M)\rightarrow L^p(M)$ be a sequence of linear operator satisfying

\begin{align}\label{Eq:FreeGroupOperator}
	&A_1. ~~  	 T_0 =\id \text{ and }\\  
\nonumber	&A_2.  T_1 T_n =\lambda T_{n+1}+(1-\lambda)T_{n-1}.
	 \text{ for all } n \in \N, \text{  where } 1/2 <  \lambda < 1.
\end{align}
We define the average of these operators as  
\begin{align*}
	M_n=\frac{1}{n+1}\sum_{r=0}^{n}T_r.
\end{align*}
In this subsection we established a maximal inequality and discuss  the convergence of the sequence $(M_n)$. 
\begin{rem}
We remark that such  examples naturally arise to study the convergence of averages of spherical averages of  free group actions.  
Let $\F_m$ be the free group with generators $\{a_1, \ldots, a_r, a_m^{-1}, \ldots, a_m^{-1}\}$. Let $\alpha: G \to Aut(M)$ denote a group homomorphism. Define,
\begin{align*}
S_n := \frac{1}{\abs{\mathbb{S}_n}} \sum_{a \in \mathbb{S}_n} \alpha(a), ~n \in \N,
\end{align*}
where, $\mathbb{S}_n$ denotes the set of words of length $n$. It is well known that $S_m \circ S_n = S_m \circ S_n$ for all $m,n \in \N$. Furthermore,
\begin{align*}
S_1 \circ S_n = \lambda S_{n+1} + (1-\lambda) S_{n-1}, ~ n \in \N,
\end{align*}
and $S_0(x)=x$ for all $x \in M$, where $1/2 < \lambda < 1$ is a fixed number.
\end{rem}

Let $\mathfrak{L}_p= \overline{\text{conv}}^{s.o.t.} \{S:L^p(M)\rightarrow L^p(M) \text{ positive Lamperti contractions}\}$.

We have the following maximal inequality of the sequence $(M_n)$. 
\begin{thm}\label{Thm:MaximalIneualityForFreeGp1}
Let $(M, \tau)$ be a noncommutative tracial measure space.  Fix $1<p<\infty$. For $n\geq 0$, let $T_n: L^p(M)\rightarrow L^p(M)$ be sequence of linear operator satisfying Eq. \eqref{Eq:FreeGroupOperator}. Assume that 	$T_1\in \mathfrak{L}_p$, then 
 there exists a constant  $C_p>0$ depending only on $p$ such that 
	\begin{align*}
		\norm{{\sup_{n}}^{+} M_n x}_{L^p(M,\ell_\infty)}\leq C_p \norm{x}_{p}.
	\end{align*}
	for all $x\in L^p(M)_+$.
\end{thm}

\begin{proof}
	For $n\geq 0$, denote 
	\begin{align*}
		S_n=\frac{1}{n+1}\sum_{r=0}^{n}T_1^r. 
	\end{align*}
	Recall that $L^p(M)$ is a ordered Banach space. 
	Notice that by \cite[Theorem 6.7]{BS23}, there exists a constant $C_w>0$ such that 
	\begin{align}\label{Eq:DominatedBySingleOp}
		M_n(x)\leq C_w S_{3n}(x), \quad\text{ for all } x\in L^p(M)_+.
	\end{align}
	Since $ 	T_1\in \mathfrak{L}_p $, apply \cite[Thm. 1.3]{HRW23} to $T_1$. Then there exists a constant $C_p^\prime>0$ $($depending only on $p)$
	such that 
	\begin{align}\label{Eq:MaximalInequalityForLamperti}
		\norm{{\sup_{n}}^{+} S_{3n}(x)}_{L^p(M,\ell_\infty)}\leq C_p^\prime \norm{x}_{p},
	\end{align} 
	for all $x\in L^p(M)_+$. 
	
	Set $C_p=C_p^\prime C_w>0$. Then, the result follows from Eqs. \eqref{Eq:DominatedBySingleOp} and \eqref{Eq:MaximalInequalityForLamperti}. This completes the proof.
\end{proof}

\begin{thm}\label{Thm:MaximalIneualityForFreeGp2}
	Let $M$ be a von Neumann algebra equipped with a $f.n.s.$ trace $\tau$. Fix $1<p, p^\prime<\infty$ such that ${1}/{p}+{1}/{p^\prime}=1$. For $n\geq 0$, let $T_n: L^p(M)\rightarrow L^p(M)$ be sequence of linear operator satisfying Eq. \eqref{Eq:FreeGroupOperator}. Assume that $T_1$ is 
	a positive Lamperti operator such that the adjoint operator $T^*: L_{p^\prime}(M)\rightarrow  L_{p^\prime}(M)$ is also Lamperti and $\sup_{n\geq 1}\norm{T_1^n}_{L^p(M)\rightarrow L^p(M)}=K<\infty$.
	Then, 
	\begin{align*}
		\norm{{\sup_{n}}^{+} M_n x}_{L^p(M,\ell_\infty)}\leq K C_p \norm{x}_{p}.
	\end{align*}
	for all $x\in L^p(M)_+$, where $C_p$ is a constant $($depending only on $p)$ appeared in Theorem \ref{Thm:MaximalIneualityForFreeGp1}.
\end{thm}
\begin{proof}
	The proof is obtained by modifying the proof of Theorem \ref{Thm:MaximalIneualityForFreeGp1} by replacing the role of \cite[Theorem 1.3]{HRW23} with \cite[Theorem 1.4]{HRW23}. This completes the proof.
\end{proof}

\begin{rem}
	We remark  that to  obtain the maximal inequality of $(M_n)$ we only assume that   $T_1 \in \mathfrak{L}_p$. Indeed, note that neither  $T_1$ falls in Dunford-Schwartz category nor there exists a f.n state $\varphi$ such that 
	$\varphi \circ T_1 \leq \varphi$ and $T_1 \circ \sigma_t^\varphi = \sigma_t^\varphi \circ  T_1$ for all $ t \in R$. Although the proofs of the obtained results are not technical, but the this results are new compare to interms of the assumption on maps (see) results obtained in  
\end{rem}

\subsection{ Ergodic theorems for spherical averages for free semigroup   actions}
Throughout this section, we assume that $\varphi$ is a faithful normal state on a von Neumann algebra $M$. In this subsection we  discuss the convergence of averages of spherical averages of free semigroup action on noncommutative probability measure space $(M, \varphi)$. We consider $I=\{1,2,\ldots, m\}$. Further, let  $j\in I$, let $\alpha_j:M\rightarrow M$ be a normal linear map and unless stated otherwise, we assume that $(\alpha_j)_{ j \in I}$  fulfils the following conditions: 
\begin{enumerate}
	\item[(4.I)]$\alpha_j$ is a contraction on $M$;
\item[(4.II)] $\alpha_j$ is completely positive;
\item[(4.III)]$\varphi\circ \alpha_j\leq \varphi$;
\item[(4.IV)] $\sigma_t^\varphi\circ \alpha_j=\alpha_j\circ \sigma_t^\varphi$;
\item[(4.V)] $\alpha_j$ is symmetric with respect to $\varphi$, i.e., $\varphi(\alpha_j(y)^* x)=\varphi(y^*\alpha_j(x))$ for all $x,y\in M$.
\end{enumerate}

For $1\leq p<\infty$, since $\{D_{\varphi}^{\frac{1}{2p}}y D_{\varphi}^{\frac{1}{2p}}:y\in M\}$ is dense in Haarerup $L^p$-space $L^p(M,\varphi)$ , where $D_\varphi$ is the density operator associated with $\varphi$, the 
map 
\begin{align*}
	D_{\varphi}^{\frac{1}{2p}}y D_{\varphi}^{\frac{1}{2p}}\rightarrow D_{\varphi}^{\frac{1}{2p}}\alpha_j(y) D_{\varphi}^{\frac{1}{2p}}, \quad y\in M,
\end{align*}
extends to a positive bounded linear map from $L^p(M,\varphi)$ to $L^p(M,\varphi)$ $($see, \cite[Theorem 5.1]{HJX10}$)$. We denote the extension by $\alpha_j$ with slight abuse of notation. Note that $\norm{\alpha_j}_{L^p(M,\varphi)\rightarrow L^p(M,\varphi)}\leq 1$.\\
Before stating our next theorem we recall the following definition.
\begin{defn}\label{Defn:AUASconvergence}
	Let $M$ be a von Neumann algebra equipped with a faithful normal state $\varphi$. 
	\begin{enumerate}
		\item A sequence $\{x_n\}_{n\geq 1}\subseteq M$ said to be converges to $x\in M$ almost uniformly $($a.u. in short$)$ $($respectively, bilaterally almost uniformly $($b.a.u. in short$))$ if for every $\epsilon>0$, there exists a projection $e\in M$ such that $\varphi(1-e)\leq \epsilon$ and $\lim_{n\rightarrow\infty} \norm{(x_n-x) e}_\infty=0$ $($respectively, $\lim_{n\rightarrow\infty} \norm{e x_n e}_\infty=0)$.
		
		\item For $1\leq p<\infty$, a sequence $\{x_n\}_{n\geq 1}\subseteq L^p(M,\varphi)$ said to be converges to $x\in L^p(M,\varphi)$ almost surely $($a.s. in short$)$ $($respectively, bilaterally almost surely $($b.a.s. in short$))$ if for every $\epsilon>0$, there exists a projection $e\in M$ and a sequence $\{a_{n,k}\}\subseteq M$ such that  $\varphi(e^\perp)\leq \epsilon$ and 
		\begin{align*}
		& x_n-x=\sum_{k}a_{n,k} D^\frac{1}{p},\quad\text{ and }\lim_{n\rightarrow\infty} \norm{\sum_{k}(a_{n,k} e)}_\infty=0\\
		(\text{respectively,  }& x_n-x=\sum_{k} (D^\frac{1}{2p}a_{n,k} D^\frac{1}{2p}),\quad\text{ and }\lim_{n\rightarrow\infty} \norm{\sum_{k}(ea_{n,k} e)}_\infty=0),
		\end{align*}
		where the two series in the above expression converges in $L^p(M,\varphi)$ and $M$
		respectively.
	\end{enumerate}
\end{defn}

\begin{thm}\label{Thm:AUconvergenceInVNA}
	Let $(M, \varphi) $ be a noncommutative probability space. For each $ i \in I$, suppose   $\alpha_i:M\rightarrow M$  a  liner map satisfying $(4.I)$-$(4.III)$. Then 
	for every $i\in I$ and $x\in M$, 
	the sequence $\left( \frac{A_n^{(i)}(x)}{p_i}\right) $
	$($defined in Eq. \eqref{Eq:AverageOfIthAverage}$)$ converges in  a.u. to an element  $\hat{x}_i$ in $M$ as $n\rightarrow\infty$. Moreover, $T(\hat{x}_1,\ldots, \hat{x}_m)=(\hat{x}_1,\ldots, \hat{x}_m)$.
	Furthermore, the sequence $\{A_n(x)\}$ defined in Eq. \eqref{Eq:AverageOfAverage}, converges a.u. to an element $\hat{x}\in M$.
\end{thm}
%{\color{red}Check: $\alpha_i(\hat{x})=\hat{x}$}

\begin{proof}
	Consider the von Neumann algebra $N=\underset{ i \in I }{\oplus}M$ and define a faithful normal state on $N$ by
	\begin{align*}
		\psi(x)=\sum_{i\in I}p_i \varphi(x_i),
	\end{align*}
	where $x=(x_i)_{i\in I}\in N$.
	
	Consider the linear map
	$T: N\rightarrow N$ defined by
	\begin{align}\label{Eq:DirectSumMatrixMap}
		T(x_1,\ldots, x_m)=(y_1\ldots,y_m), \quad\text{ where } y_j=\sum_{i\in I}\frac{p_i p_{ij}}{p_j} \alpha_j (x_i), 
	\end{align}
	where $x=(x_i) \in N$. Note that the map $T$ defined in Eq. \eqref{Eq:DirectSumMatrixMap} is a normal positive contraction satisfy $\psi\circ T\leq \psi$.
	Thus, by a result of K\"{u}mmerer \cite[Theorem 4.5]{Ku78}, it follows that
	for each $x\in N$, the average $\frac{1}{n} \underset{j=0}{ \overset{n-1}{\sum}} T^j(x)$ converges in a.u. to a $T$-invariant element $\hat{x}\in N$. Thus, for $i\in X$, the sequence $\left( \frac{A_n^{(i)}(x)}{p_i}\right)$ defined in Eq. \eqref{Eq:AverageOfIthAverage}, converges in a.u. to an element $\hat{x}_i$. Consequently,
	$\hat{x}=(\hat{x}_1,\ldots, \hat{x}_m)$. Since $T(\hat{x})=\hat{x}$, we have 
	$T(\hat{x}_1,\ldots, \hat{x}_m)=(\hat{x}_1,\ldots, \hat{x}_m)$.
	
For the last part, since $A_n= \underset{i=1}{ \overset{m}{\sum}}  A_n^i$, so, the sequence $\left( A_n(x)\right) $ converges in  a.u. to an element $\hat{x}\in M$. This completes the proof.
\end{proof}

\begin{thm}\label{Thm:MultiParameterErgodicTheoremInLp}
Let $(M, \varphi) $ be noncommutative probability space. For each $i \in I$,  let   $\alpha_i: M\rightarrow M$ be a    liner map satisfying $(4.I)$-$(4.IV)$. Let $1\leq p<\infty$. Then we have the following results.
\begin{enumerate}
	\item 
 For every $i\in X$ and $x\in L^p(M,\varphi)$, 
	the sequence $\left( \frac{A_n^{(i)}(x)}{p_i}\right) $
	converges to an element $\hat{x}_i\in L^p(M,\varphi)$ as $n\rightarrow\infty$ in b.a.s.. 
	Moreover, we have  $T(\hat{x}_1,\ldots, \hat{x}_m)=(\hat{x}_1,\ldots, \hat{x}_m)$.
	
	\item  If $p>2$, the convergence in $(i)$ is a.s. as well.
	\item  For $p>1$, and $x\in L^p(M,\varphi)$, the sequence  $\left( \frac{A_n^{(i)}(x)}{p_i}\right) $ converges in $L^p$-norm.
	
	\item  The results in $(1)$-$(3)$ (Theorem \ref{Thm:MultiParameterErgodicTheoremInLp}) hold for the sequence $\{A_n(x)\}$ as well. 
	
\end{enumerate}

\end{thm}

\begin{proof}
of $(1)-(2)$.	Let $N$, $\psi$ and $T$ be as in the proof of Theorem \ref{Thm:AUconvergenceInVNA}. Observe that 
	$T$ is completely positive such that $\psi\circ T\leq \psi$. Note that $\sigma_t^\psi= \underset{i\in I}{\oplus} \sigma_t^\varphi$, and hence $T\circ \sigma_t^\psi=\sigma_t^\psi\circ T$, $t\in\mathbb{R}$. Now for $p>1$, apply a result of Junge-Xu \cite[Corollary 7.12]{JX07} and for $p=1$, apply  a remark of Junge-Xu \cite[pp 429]{JX07} to the average $\frac{1}{n} \underset{  j=0}{\overset{n-1}{\sum}}T^j $ and then for all $  x \in L^p(N,\psi)$,  we obtain that the averaging sequence $\left(\frac{1}{n} \underset{  j=0}{\overset{n-1}{\sum}} T^j(x)\right)$ converges to a $T$-invariant element $\hat{x}$ in $b.a.s.$. Moreover, if $p>2$, the above convergence is in a.s.

	It is known that $L^p(N,\psi)$ is identified with $\underset{i\in I}{\oplus}  L^p(M,\varphi)$ with norm defined as  
	$$\norm{(x_1,x_2,\ldots,x_n)}=(\sum_{i\in I}p_i\norm{x_i}_{p}^p)^{1/p},$$
	where $(x_1,x_2,\ldots,x_n)\in  \underset{i\in I}{\oplus} L^p(M,\varphi)$ \cite[Remark 30]{Te81}. For $x\in L^p(M,\varphi)$, we consider  $(x,\ldots,x)\in L^p(N,\psi)$. Then  apply  the discussion in the above paragraph and  it follows that if $x\in L^p(M,\varphi)$, 
	then the sequence $\left( \frac{A_n^{(i)}(x)}{p_i}\right) $, defined in Eq. \eqref{Eq:AverageOfIthAverage}, converges b.a.s. to an 
	element $\hat{x}_i\in L^p(M,\varphi)$. Moreover, if $p>2$, the above convergence is in a.s. Furthermore, we have $$T(\hat{x}_1,\ldots, \hat{x}_m)=(\hat{x}_1,\ldots, \hat{x}_m).$$

	\emph{ Proof. of $(3)$}. For $p>1$, apply Lorch's erogodic theorem $($see \cite[Theorem 2]{Lo39}, \cite[Theorem 1.2, pp 73]{Kr85}$)$ to $T$ $(T$ is being a power bounded operator on a reflexive Banach space$)$ and obtain  the convergence of $\left( \frac{A_n^{(i)}(x)}{p_i}\right) $ in $L^p(M,\varphi)$ in norm. 
	
The proof of part $(3)$ follows as we have $A_n=\underset{i \in I}{ \sum} A_n^i$. This completes the proof.
	%
	%
	%Check $L^p(N,\psi)=L^p(M,\varphi)\oplus L^p(M,\varphi)$ \cite[Remark 30]{Te81} and the calculation for above statement.
	%
	%
	%Check for $p=1$ goldstein \cite{Go81} and Remark of Junge-Xu \cite[pp 429]{JX07}
\end{proof}

\begin{thm}\label{Thm:MultiParameterMaximalInequalityInLp}
	Let $(M, \varphi) $ be noncommutative probability space. For each $ i \in I$, let $\alpha_i:M\rightarrow M$,  be a liner map satisfying $(4.I)$-$(4.V)$ and suppose  $1<p<\infty$. Then there exist constants $C_p^{(i)}>0$, $i\in I$, and $C_p>0$ $($depending only on $p)$ such that 
	\begin{align*}
		\norm{{\sup}_{n}^{+} A_n^{(i)}(x)}_{L^p(M,\ell_\infty)}&\leq C_p^{(i)} \norm{x}_p, \quad x\in L^p(M,\varphi)\\
		\norm{{\sup}_{n}^{+} A_n(x)}_{L^p(M,\ell_\infty)}&\leq C_p \norm{x}_p, \quad x\in L^p(M,\varphi).
	\end{align*}
\end{thm}
\begin{proof}
	Let $N$, $\psi$ and $T$ be as in the proof of Theorem \ref{Thm:AUconvergenceInVNA}. Then $T$ satifies the conditions (7.I)-(7.IV) of Junge-Xu \cite{JX07}. Now apply maximal ergodic inequality of Junge-Xu \cite[Theorem 7.4]{JX07} to the average of the operator $T$, the result follows.
\end{proof}

Now we will establish the invariance of the limit that  obtained in  Theorem \ref{Thm:MultiParameterErgodicTheoremInLp} . We  follow similar  approach as in the works of \cite{GKS07} and \cite{Bu00}. We point out that the   main technique was first introduced by \cite{Gui69}. Before stating the main result of this section, let us recall an elementary result that is commonly known.

\begin{lem}\label{Lem:StrictlyConvex}
	Let $ X$ be a strictly convex normed linear space. Then the following hold:
	\begin{enumerate}
		\item  If $\alpha: X\rightarrow  X$ is a contraction, and $x,y\in  X$ such that  $\norm{x}=\norm{y}=\norm{\alpha(x)}=\norm{\alpha(y)}$ and $\alpha(x)=\alpha(y)$, then $x=y$.
	
	\item Let $x, x_1,\ldots, x_n\in  X$ be such that $x=\underset{i =1}{ \overset{n}{\sum}} c_i x_i$ for some $c_1,\ldots, c_n>0$ such that $\underset{i =1}{ \overset{n}{\sum}}  c_i=1$. If $\norm{x}=\norm{x_1}=\cdots=\norm{x_n}$, then $x=x_1=\cdots=x_n$.	
	\end{enumerate}
\end{lem}
\begin{proof} \emph { of $(1)$.} We prove this result by contradiction.	Without any  loss of generality, one can assume that $\norm{x}=\norm{y}=1$. Suppose $x\neq y$. Since $ X$ is strictly convex,  we have $\norm{\frac{x+y}{2}}<1$. Note that
\begin{align*}
	1=\norm{x}=\norm{\alpha(x)}
	&=\norm{\frac{\alpha(x)+\alpha(y)}{2}}\quad (\text{since }\alpha(x)=\alpha(y))\\
	&=\norm{\alpha(\frac{x+y}{2})}
	\leq \norm{\frac{x+y}{2}}
	<1.
\end{align*}
This is a contradiction to the hypothesis. Thus, $x=y$. \\

\emph{ Proof. of $(2)$. } The proof of this fact is standard and well-known. Thus, we omit the details. This completes the proof.
\end{proof}

\begin{thm}
Let $(M, \varphi) $ be noncommutative probability space. For each $i \in I$,  let   $\alpha_i: M\rightarrow M$ be a  liner map satisfying $(4.I)$-$(4.IV)$  and suppose  $2\leq p<\infty$. For $x\in L^p(M,\varphi)$, let $\hat{x}_i\in L^p(M,\varphi)$, $i\in I$, be as in Theorem \ref{Thm:MultiParameterErgodicTheoremInLp}, that satisfies the condition
\begin{align*}
	T(\hat{x}_1,\ldots, \hat{x}_m)=(\hat{x}_1,\ldots, \hat{x}_m).
\end{align*} 
If $\mu$ is a strictly irreducible $\sigma_m$-invariant Markov measure on $I$, then
 $$\hat{x}_1=\cdots=\hat{x}_m:=\hat{x} \text{ and }\alpha_i(\hat{x})=\hat{x}, \text{ for all }i\in I.$$
\end{thm}
\begin{proof}
First suppose $\hat{x}_i\in L^p(M,\varphi)$, $i\in I$, be such that $T(\hat{x}_1,\ldots, \hat{x}_m)=(\hat{x}_1,\ldots, \hat{x}_m)$. In this case, we prove that there exists $l>0$ such that
$\norm{\hat{x}_j}_p=l$, for all $j\in I$, and $\norm{\alpha_j (\hat{x}_i)}_p=l$ if $p_{ij}>0$,  for $i,j\in I$. \\
Indeed, by Eq. \eqref{Eq:DirectSumOperator} and from the hypothesis, we have
\begin{align}\label{Eq:LimitEquality}
	\hat{x}_j=\sum_{i\in I}\frac{p_i p_{ij}}{p_j} \alpha_j (\hat{x}_i),\quad j\in I.
\end{align}
Let $j\in I$ be such that $\norm{\hat{x}_j}_p=\max_{i\in I}	\norm{\hat{x}_i}_p$. By Eq. \eqref{Eq:LimitEquality}, we have
\begin{align*}
	\norm{\hat{x}_j}_p
	&\leq\sum_{i\in I}\frac{p_i p_{ij}}{p_j} \norm{\alpha_j (\hat{x}_i)}_p\\
	&\leq \sum_{i\in I}\frac{p_i p_{ij}}{p_j} \norm{ \hat{x}_i}_p \quad(\text{since }\alpha_j\text{ are contraction})\\
	&\leq \norm{\hat{x}_j}_p \sum_{i\in I}\frac{p_i p_{ij}}{p_j}\quad(\text{since }\norm{\hat{x}_j}_p\geq	\norm{\hat{x}_i}_p,\forall i\in I)\\
	&=\norm{\hat{x}_j}_p.
\end{align*}
Thus, equality in the above equation is everywhere. Hence,
\begin{align*}
	\norm{\hat{x}_j}_p=	\norm{ (\hat{x}_i)}_p=\norm{\alpha_j (\hat{x}_i)}_p, \quad \text{ for all }i\in I \text{ such that }\ p_{ij}\neq 0.
\end{align*}	

Similarly, by Eq. \eqref{Eq:LimitEquality}, one has
\begin{align}\label{Eq:LimitEquality2}
	\hat{x}_j
    \nonumber &=\sum_{i\in I}\frac{p_i p_{ij}}{p_j} \alpha_j (\hat{x}_i)\\
	\nonumber &=\sum_{i\in I}\frac{p_i p_{ij}}{p_j} \alpha_j \left(\sum_{k\in I}\frac{p_k p_{ki}}{p_i}\alpha_i(\hat{x}_k)\right)\\
	 &=\sum_{k\in I}\frac{p_{k}}{p_j} \left(\sum_{i\in I}p_{ki}p_{ij}
	 \alpha_j(\alpha_i(\hat{x}_k))\right).
\end{align}
Observe that $(p_1,\ldots, p_m)P^2=(p_1,\ldots, p_m)$. Thus, repeat the same argument as presented in the previous paragraph and obtain  $\norm{\hat{x}_j}_p=\norm{\hat{x}_i}_p$, if 
$p_{ij}^{(2)}\neq 0$. Since $P$ is irreducible, for each $i\in I$, there exists $n\in \mathbb{N}$ such that $p_{ij}^{(n)}\neq 0$. Now apply the same argument as before, we have
\begin{align*}
	\norm{\hat{x}_j}_p=\norm{\hat{x}_i}_p, \forall i,j\in I.\\
\end{align*}

Now  we show that if $(PP^t)_{ij}>0$, then $\hat{x}_i=\hat{x}_j$.
Indeed, since $(PP^t)_{ij}>0$, there exists $k\in I$ such that $p_{ik}>0$ and $p_{jk}>0$. Note that $\hat{x}_k$ is the convex combinations of $\alpha_k (\hat{x}_m)$, i.e., 
\begin{align*}
\hat{x}_k=\sum_{m\in I}c_m \alpha_k (\hat{x}_m),
\end{align*}
where $c_m=\frac{p_m p_{mk}}{p_k}\geq 0$, $m\in I$.
Since $c_i,c_j>0$ and $L^p(M,\varphi)$ is strictly convex for $2\leq p<\infty$, by Lemma \ref{Lem:StrictlyConvex}$(ii)$, one has
\begin{align*}
	\hat{x}_k=\alpha_k(\hat{x}_i)=\alpha_k(\hat{x}_j).
\end{align*}
Moreover, by first part
\begin{align*}
	\norm{\hat{x}_k}=\norm{\alpha_k(\hat{x}_i)}=\norm{\alpha_k(\hat{x}_j)}=\norm{\hat{x}_i}=\norm{\hat{x}_j}
\end{align*}
Now apply Lemma \ref{Lem:StrictlyConvex}$(i)$ and obtain  $	\hat{x}_i=\hat{x}_j$. This proves the claim.\\
 Finally,  we show that
\begin{align*}
	\hat{x}_1=\cdots=\hat{x}_m.
\end{align*}
Indeed, since $\mu$ is strictly irreducible, there exists $n\in\mathbb{N}$ such that $((PP^t)^n)_{ij}>0$. By applying the similar argument as previous case, we have $\hat{x}_1=\cdots=\hat{x}_m$. This completes the proof.
\end{proof}

%%%%%%%%%%%%%%%%%%%%%%%%%%%%%%%%%%%%%%%%%%%%%%%%%%%%%%%%%%%%%%%%%%%%%%%%%%%%%%%%%%%%%%%%%%%%%%%%%%%%%%%%%%%%%%%%%%%%%%%%%%%%%%%%%%%%%%%%%%%%%%%%%%%%%%%%%%%%%%%%%%%%%%%%%%%%%%%%%%%%%%%%%%%%%%%%%%%%%%%%%%%%%%%%%%%%%%%%%%%%%%%%%%%%%%%%%%%%%%%%%%%%%%%%%%%%%%%%%%%%%%%%%%%%%%%%%%%%%%%%%%%%%%%%%%%%%%%%%%%%%%%%%%%%%%%%%%%%%%%%%%%%%%%%%%%%%%%%%%%%%%%%%%%%%%%%%%%%%%%%%%%%%%%%%%%%%%%%%%%%%%%%%%%%%%%%%%%%%%%%%%%%%%%%%%%%%%%%%%%%%%%%%%%%%    Appendix %%%%%%%%%%%%%%%%%%%%%%%%%%%%%%%%%%%%%%%%%%%%%%%%%%%%%%%%%%%%%%%%%%%%%%%%%%%%%%%%%%%%%%%%%%%%%%%%%%%%%%%%%%%%%%%%%%%%%%%%%%%%%%%%%%%%%%%%%%%%%%%%%%%%%%%%%%%%%%%%%%%%%%%%%%%%%%%%%%%%%%%%%%%%%%%%%%%%%%%%%%%%%%%%%%%%%%%%%%%%%%%%%%%%%%%%%%%%%%%%%%%%%%%%%%%%%%%%%%%%%%%%%%%%%%%%%%%%%%%%%%%%%%%%%%%%%%%%%%%%%%%%%%%%%%%%%%%%%%%%%%%%%%%%%%%%%%%%%%%%%%%%%%%%%%%%%%%%%%%%%%%%%%%%%%%%%%%%%%%%%%%%%%%%%%%%%%%%%%%%%%%%%%%%%%%%%%%%%%%%%%%%%%%%%%%%%%%%%%%%%%%%%%%%%%%%%%%%%%%%%%%%%%%%%%%%%%%%%%%%%%%%%%%%%%%%%%%%%%%%%%%%%%%%%%%%%%%%%%%%%%%%%%%%%%%%%%%%%%%%%%%%%%%%%%%%%%%%%%%%%%%%%%%%%%%%%%%%%%%%%%%%%%%%%%%%%%%%%%%%%%%%%%%%%%%%%%%%%%%%%%%%%%%%%%%%%%%%%%%%%%%%%%%%%%%%%%%%%%%%%%%%%%%%%%%%%%%%

%\bibliographystyle{amsalpha}
%\bibliography{free}

%\end{document}

\bibliographystyle{amsalpha}

\end{document}